\documentclass[12pt,oneside]{amsart}
\usepackage{enumitem}
\usepackage{wrapfig}
\usepackage{caption}
\usepackage{amsthm, amsmath}
\usepackage[charter]{mathdesign}
\usepackage{color}

\usepackage{tikz}
  \usetikzlibrary{decorations.pathreplacing}
  \usetikzlibrary{arrows,decorations.pathmorphing}
  \tikzset{snake it/.style={decorate, decoration={snake, amplitude = 2, segment length = 5,}}}
  \usetikzlibrary{decorations.pathreplacing}
  \usetikzlibrary{arrows,decorations.pathmorphing}
  \usetikzlibrary{fit}
  \tikzstyle{arrow}=[draw, -latex]
  \tikzset{>=latex}

\usepackage{geometry}
\setlist{leftmargin=1cm}

\allowdisplaybreaks

\newtheorem{theorem}{Theorem}

\newtheorem{lemma}[theorem]{Lemma}
\newtheorem{remark}[theorem]{Remark}

\theoremstyle{definition}

\begin{document}
\title{Combinatorics arising from lax colimits of posets}
\author{Zurab Janelidze, Helmut Prodinger, and Francois van Niekerk}
\begin{abstract}
In this paper we study maximal chains in certain lattices constructed from powers of chains by iterated lax colimits in the $2$-category of posets. Such a study is motivated by the fact that in lower dimensions, we get some familiar combinatorial objects such as Dyck paths and Kreweras walks.      
%In this paper we show that the odd double factorial $(2n-1)!!$ counts maximal chains in the lattice obtained as the lax colimit, in the $2$-category of posets, of a sequence of embedded hypercubes up to dimension $n$, i.e., lattices $C_1^j$, where $C_1$ denotes the $1$-chain and $j\leqslant n$. We then generalize this result to the case when $C_1$ is replaced with an $m$-chain $C_m$, and obtain some results about iterated lax colimits of the corresponding sequences, including a representation of the $k$-iterated lax colimit as a sublattice of $C^n_{k+m}$. In the case when $m=0$, the corresponding maximal chain numbers are the $k$-dimensional Catalan numbers. 
\end{abstract}
\maketitle

\section*{Introduction}

The well-known combinatorial objects, Dyck paths, can be interpreted as maximal chains in lattices (in the sense of \cite{Bir48}) given by the following Hasse diagrams:

$$
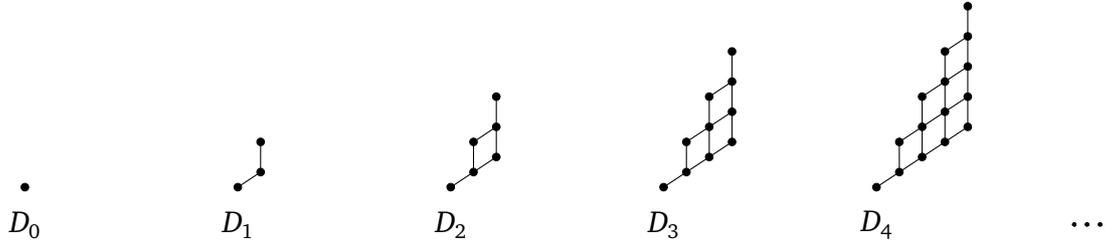
\begin{tikzpicture}
  \pgfmathsetmacro{\hx}{2.8}
  \pgfmathsetmacro{\hy}{0}
  \pgfmathsetmacro{\ix}{0.3}
  \pgfmathsetmacro{\iy}{0.2}
  \pgfmathsetmacro{\kx}{0}
  \pgfmathsetmacro{\ky}{0.4}
		
  \foreach \h in {0,...,4} {
    \node at(\hx*\h,\hy*\h-0.5){$D_{\h}$};
	\node at(\hx*5,\hy*5-0.5){$\dots$};
		\tikzset{every node/.style={shape=circle,draw=black,fill=black,inner sep=1pt, outer sep=0pt}}	
		\foreach \i in {0,...,\h} {
				\foreach \k in {0,...,\i} {
					\node (\i\k) at(\hx*\h+\ix*\i+\kx*\k,\hy*\h+\iy*\i+\ky*\k){};
					\ifthenelse{\i=\h}{}{\draw (\hx*\h+\ix*\i+\kx*\k,\hy*\h+\iy*\i+\ky*\k) -- (\hx*\h+\ix*\i+\ix+\kx*\k,\hy*\h+\iy*\i+\iy+\ky*\k);}
					\ifthenelse{\k=\i}{}{\path (\hx*\h+\ix*\i+\kx*\k,\hy*\h+\iy*\i+\ky*\k) edge (\hx*\h+\ix*\i+\kx*\k+\kx,\hy*\h+\iy*\i+\ky*\k+\ky);}
					}
			}
	}
\end{tikzpicture}
$$

\noindent Each of the lattices $D_m$ can be constructed as a `lax colimit', in the $2$-category of posets, of the diagram

$$
	\begin{tikzpicture}
	\foreach \k in {0,...,6} {
		\node (\k) at(2*\k,0){\ifthenelse{\k<5}{$C_\k$}{\ifthenelse{\k=5}{$\dots$}{\ifthenelse{\k=6}{$C_m$}{}}}};
	}
	\foreach \k in {1,...,6} {
		\pgfmathsetmacro{\l}{\k-1}
		\draw [<-] (\k) edge (\l);
	}
	\end{tikzpicture}
$$

\noindent where $C_i$ stands for the chain with $i+1$ elements and each homomorphism $C_i\to C_{i+1}$ is an inclusion of down-closed sub-join-semilattices. A lax colimit `stacks' the chains above each other, turning each assignment $x\mapsto y$ given by a homomorphism into the relation $x<y$ in the lax colimit. This process can be visualized in the case of the diagram above with $n=3$ as follows: 

$$
		
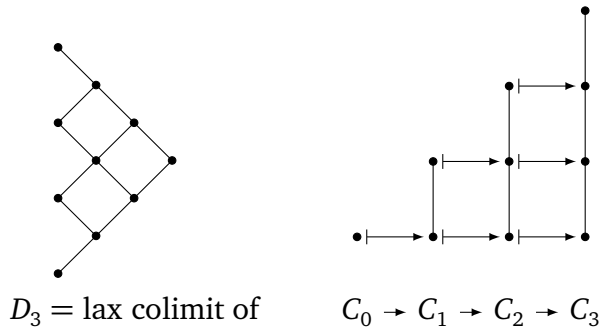
\begin{tikzpicture}[scale=0.5]
			\node at(2,-1){$D_{3}=$ lax colimit of };
			\tikzset{every node/.style={shape=circle,draw=black,fill=black,inner sep=1pt}}
			\foreach \i in {0,...,3} {
				\foreach \j in {0,...,\i} {
					\node at(\j,2*\i-\j){};
					\ifthenelse{\i<3}{\draw (\j,2*\i-\j) -- (\j+1,2*\i-\j+1);}{}
					\ifthenelse{\j<\i}{\draw (\j,2*\i-\j) -- (\j+1,2*\i-\j-1);}{}
				}
			}
		\end{tikzpicture}
		\quad\quad
		\begin{tikzpicture}[scale=1]
			\foreach \i in {0,...,3} {
			\node (label\i) at(\i,-1){$C_{\i}$};
			\tikzset{every node/.style={shape=circle,draw=black,fill=black,inner sep=1pt,outer sep=2pt}}
				\foreach \j in {0,...,\i} {
					\node (\i\j) at(\i,\i-\j){};
					\ifthenelse{\j<\i}{\draw (\i,\i-\j) -- (\i,\i-\j-1);}{}
				}
			}
			\draw [|->] (00) -- (11);
			\draw [|->] (11) -- (22);
			\draw [|->] (22) -- (33);
			\draw [|->] (10) -- (21);
			\draw [|->] (21) -- (32);
			\draw [|->] (20) -- (31);
			\draw [->] (label0) -- (label1);
			\draw [->] (label1) -- (label2);
			\draw [->] (label2) -- (label3);			
		\end{tikzpicture}
$$

\noindent One may now consider variations of this construction, where chains are replaced with other lattices. For instance, when we replace each $C_i$ with its cartesian square $C^2_i$, we get the following lattices:

$$
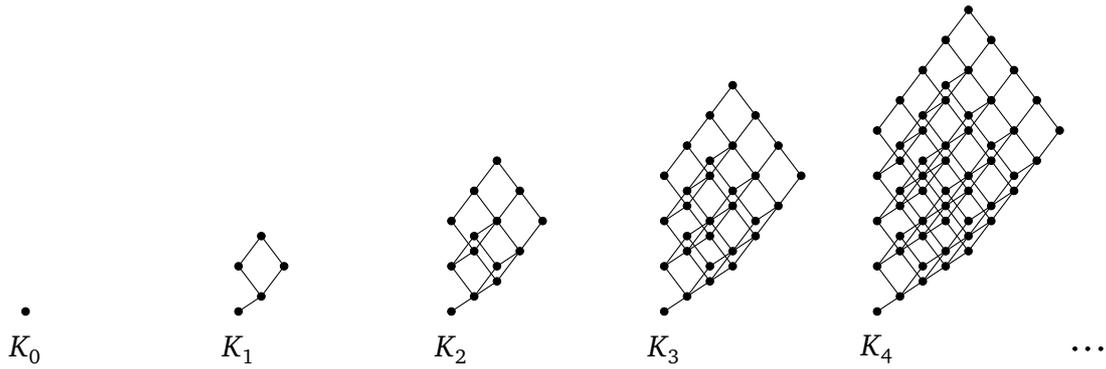
\begin{tikzpicture}
		\pgfmathsetmacro{\hx}{2.8}
		\pgfmathsetmacro{\hy}{0}
		\pgfmathsetmacro{\ix}{0.3}
		\pgfmathsetmacro{\iy}{0.2}
		\pgfmathsetmacro{\jx}{-0.3}
		\pgfmathsetmacro{\jy}{0.4}
		\pgfmathsetmacro{\kx}{0.3}
		\pgfmathsetmacro{\ky}{0.4}
	\foreach \h in {0,...,4} {
		\node at(\hx*\h,\hy*\h-0.5){$K_{\h}$};
		\node at(\hx*5,\hy*5-0.5){$\dots$};
		\tikzset{every node/.style={shape=circle,draw=black,fill=black,inner sep=1,outer sep=2pt}}			
		\foreach \i in {0,...,\h} {
			\foreach \j in {0,...,\i} {
				\foreach \k in {0,...,\i} {
					\node (\i\j\k) at(\hx*\h+\ix*\i+\jx*\j+\kx*\k,\hy*\h+\iy*\i+\jy*\j+\ky*\k){};
					\ifthenelse{\i=\h}{}{\draw (\hx*\h+\ix*\i+\jx*\j+\kx*\k,\hy*\h+\iy*\i+\jy*\j+\ky*\k) -- (\hx*\h+\ix*\i+\ix+\jx*\j+\kx*\k,\hy*\h+\iy*\i+\iy+\jy*\j+\ky*\k);}
					\ifthenelse{\j=\i}{}{\draw (\hx*\h+\ix*\i+\jx*\j+\kx*\k,\hy*\h+\iy*\i+\jy*\j+\ky*\k) -- (\hx*\h+\ix*\i+\jx*\j+\jx+\kx*\k,\hy*\h+\iy*\i+\jy*\j+\jy+\ky*\k);}
					\ifthenelse{\k=\i}{}{\draw (\hx*\h+\ix*\i+\jx*\j+\kx*\k,\hy*\h+\iy*\i+\jy*\j+\ky*\k) -- (\hx*\h+\ix*\i+\jx*\j+\kx*\k+\kx,\hy*\h+\iy*\i+\jy*\j+\ky*\k+\ky);}
					}
				}
			}
	}
\end{tikzpicture}
$$

\noindent Each of these lattices decomposes as a lax colimit as follows (the example shown is for $n=3$):

$$
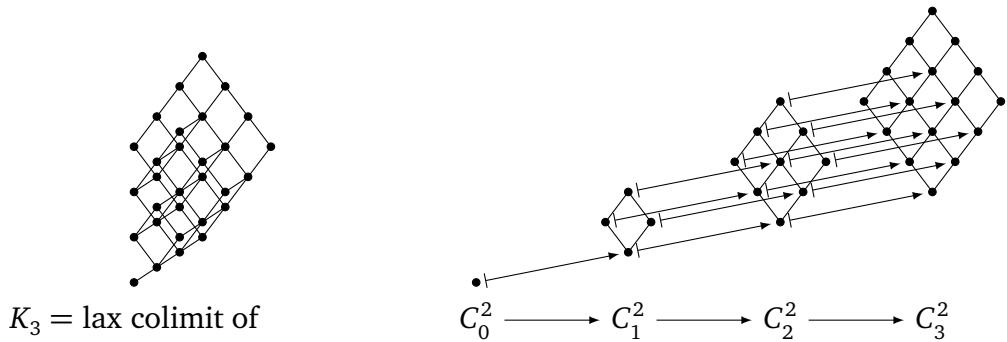
\begin{tikzpicture}

		\pgfmathsetmacro{\hx}{0}
		\pgfmathsetmacro{\hy}{0}
		\pgfmathsetmacro{\ix}{0.3}
		\pgfmathsetmacro{\iy}{0.2}
		\pgfmathsetmacro{\jx}{-0.3}
		\pgfmathsetmacro{\jy}{0.4}
		\pgfmathsetmacro{\kx}{0.3}
		\pgfmathsetmacro{\ky}{0.4}
	\foreach \h in {3} {
		\node at(\hx*\h,\hy*\h-0.5){$K_{\h}=$ lax colimit of };
		\tikzset{every node/.style={shape=circle,draw=black,fill=black,inner sep=1pt,outer sep=2pt}}		
		\foreach \i in {0,...,\h} {
			\foreach \j in {0,...,\i} {
				\foreach \k in {0,...,\i} {
					\node at(\hx*\h+\ix*\i+\jx*\j+\kx*\k,\hy*\h+\iy*\i+\jy*\j+\ky*\k){};
					\ifthenelse{\i=\h}{}{\draw (\hx*\h+\ix*\i+\jx*\j+\kx*\k,\hy*\h+\iy*\i+\jy*\j+\ky*\k) -- (\hx*\h+\ix*\i+\ix+\jx*\j+\kx*\k,\hy*\h+\iy*\i+\iy+\jy*\j+\ky*\k);}
					\ifthenelse{\j=\i}{}{\draw (\hx*\h+\ix*\i+\jx*\j+\kx*\k,\hy*\h+\iy*\i+\jy*\j+\ky*\k) -- (\hx*\h+\ix*\i+\jx*\j+\jx+\kx*\k,\hy*\h+\iy*\i+\jy*\j+\jy+\ky*\k);}
					\ifthenelse{\k=\i}{}{\draw (\hx*\h+\ix*\i+\jx*\j+\kx*\k,\hy*\h+\iy*\i+\jy*\j+\ky*\k) -- (\hx*\h+\ix*\i+\jx*\j+\kx*\k+\kx,\hy*\h+\iy*\i+\jy*\j+\ky*\k+\ky);}
					}
				}
			}
	}

		\pgfmathsetmacro{\hx}{1.5}
		\pgfmathsetmacro{\hy}{0}
		\pgfmathsetmacro{\ix}{2}
		\pgfmathsetmacro{\iy}{0.4}
		\pgfmathsetmacro{\jx}{-0.3}
		\pgfmathsetmacro{\jy}{0.4}
		\pgfmathsetmacro{\kx}{0.3}
		\pgfmathsetmacro{\ky}{0.4}
	\foreach \h in {3} {
		\foreach \i in {0,...,\h} {
			\node (label\i) at(\hx*\h+\ix*\i,\hy*\h-0.5){$C^2_{\i}$};	
			\tikzset{every node/.style={shape=circle,draw=black,fill=black,inner sep=1pt,outer sep=2pt}}				
			\foreach \j in {0,...,\i} {
				\foreach \k in {0,...,\i} {
					\node (\i\j\k) at(\hx*\h+\ix*\i+\jx*\j+\kx*\k,\hy*\h+\iy*\i+\jy*\j+\ky*\k){};
					\ifthenelse{\i=\h}{}{}
					\ifthenelse{\j=\i}{}{\draw (\hx*\h+\ix*\i+\jx*\j+\kx*\k,\hy*\h+\iy*\i+\jy*\j+\ky*\k) -- (\hx*\h+\ix*\i+\jx*\j+\jx+\kx*\k,\hy*\h+\iy*\i+\jy*\j+\jy+\ky*\k);}
					\ifthenelse{\k=\i}{}{\draw (\hx*\h+\ix*\i+\jx*\j+\kx*\k,\hy*\h+\iy*\i+\jy*\j+\ky*\k) -- (\hx*\h+\ix*\i+\jx*\j+\kx*\k+\kx,\hy*\h+\iy*\i+\jy*\j+\ky*\k+\ky);}
					}
				}
			}
		}
		
		\draw [->] (label0) -- (label1);
		\draw [->] (label1) -- (label2);
		\draw [->] (label2) -- (label3);
		\draw [|->] (000) -- (100);
		\draw [|->] (100) -- (200);
		\draw [|->] (200) -- (300);
		\draw [|->] (101) -- (201);
		\draw [|->] (201) -- (301);
		\draw [|->] (202) -- (302);
		\draw [|->] (110) -- (210);
		\draw [|->] (210) -- (310);
		\draw [|->] (220) -- (320);
		\draw [|->] (111) -- (211);
		\draw [|->] (211) -- (311);
		\draw [|->] (212) -- (312);
		\draw [|->] (221) -- (321);
		\draw [|->] (222) -- (322);

\end{tikzpicture}
$$

\noindent Maximal chains in these lattices are in bijection with Kreweras walks \cite{BouMel05,Kre65} --- an observation originally due to Sarah Selkirk (private communication). This gives rise to the following question: \emph{which combinatorial objects arise as maximal chains in lattices stacked by means of lax colimits?} For instance, noting that the two examples above correspond to the second and the third rows in a (commutative) diagram of powers of chains (see Figure~\ref{FigA}), it becomes interesting to explore other sequences of homomorphisms arising from the same diagram.

Stacking lattices in the first row of the diagram in Figure~\ref{FigA} simply gives the sequence of chains (the second row) --- not so interesting in its own right. However, stacking lattices in the sequence of chains, so lattices in the second row, gives the Dyck situation described above. 
Maximal chains in these lattices (i.e., the lattices $D_m$) can be counted, as it is well known, by the Catalan numbers:
$$\frac{(2m)!}{m!(m+1)!}$$
The third row gives the Kreweras case. This already is a highly nontrivial combinatorial situation; for example, a bijective proof for the formula 
$$\frac{(3m)!4^m}{(m+1)!(2m+1)!}$$
that counts Kreweras walks (i.e., maximal chains in the lattices $K_m$) was found not long ago \cite{Ber07}. This might mean that getting similar numbers for stacking lattices along the subsequent rows can be vastly difficult. In this paper we solve the orthogonal problem: we find the maximal chain numbers for stacking lattices along the columns of the diagram in Figure~\ref{FigA}. The first column is again trivial, as it is identical (or rather, isomorphic) to the first row. The sequence of number here is simply the sequence of all positive natural numbers $1,2,3,4,\dots$. The second column stacks `hypercubes'. We show that the corresponding numbers are given by the odd double factorials:
$$\frac{(2n)!}{2^n n!}=(2n-1)!!$$ 
We first prove this by making use of the technique of representing weighted Dyck paths as involutions with no fixed points, from the lecture series of X. Viennot \cite{Vie06}, which then suggests a direct bijection with maximal chains in stacks of hypercubes. After this, it becomes evident how to deal with the remaining columns: involutions with no fixed points, which are the same as $2$-partitions, get replaced with $m$-partitions (of an $mn$ element set), giving us the numbers
$$\frac{(mn)!}{(m!)^n n!}$$
for maximal chains in the stacks along the $m$-th column of Figure~\ref{FigA}. 

\begin{figure}
\begin{tikzpicture}
		\pgfmathsetmacro{\ix}{1.3}
		\pgfmathsetmacro{\iy}{0}
		\pgfmathsetmacro{\jx}{0}
		\pgfmathsetmacro{\jy}{-1.3}
		\foreach \i in {0,...,4} {
			\foreach \j in {0,...,4} {
					\node (\i\j) at(\ix*\i+\jx*\j,\iy*\i+\jy*\j){$C^{\j}_{\i}$};
				}
			}
		\foreach \i in {5} {
			\foreach \j in {0,...,4} {
					\node (\i\j) at(\ix*\i+\jx*\j,\iy*\i+\jy*\j){$\dots$};
				}
			}
		\foreach \i in {0,...,4} {
			\foreach \j in {5} {
					\node (\i\j) at(\ix*\i+\jx*\j,\iy*\i+\jy*\j){$\vdots$};
				}
			}
		\foreach \i in {0,...,4} {
			\draw [->] (\i0) -- (\i1);
			\draw [->] (\i1) -- (\i2);
			\draw [->] (\i2) -- (\i3);
			\draw [->] (\i3) -- (\i4);
			\draw [->] (\i4) -- (\i5);
			\draw [->] (0\i) -- (1\i);
			\draw [->] (1\i) -- (2\i);
			\draw [->] (2\i) -- (3\i);
			\draw [->] (3\i) -- (4\i);
			\draw [->] (4\i) -- (5\i);

		}
\end{tikzpicture}
\caption{Homomorphism lattice of powers of chains}\label{FigA}
\end{figure}
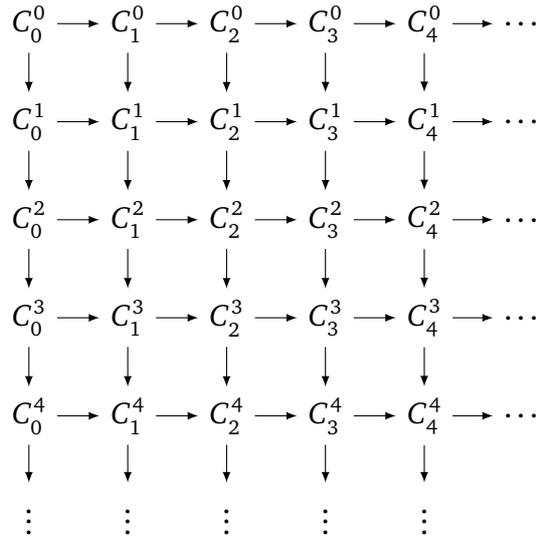

These results are obtained in  Section~\ref{SecA} of this paper (see Theorems~\ref{ThmG} and \ref{ThmA}). In Section~\ref{SecB}, we turn our attention to iterating the process of stacking and find the following: 
\begin{itemize}
\item Maximal chain numbers for $k$-th iteration of stacking lattices in the first column of Figure~\ref{FigA} are the $k$-dimensional Catalan numbers from \cite{SnoTro89}:
$$\frac{(k-1)!(kn)!}{n!(n+1)!\dots(n+k-1)!}.$$

\item Maximal chains in $k$-th iteration of stacking hypercubes (i.e., lattices in the second column of Figure~\ref{FigA}) are in bijection with the combinatorial objects discussed in \cite{FucYuZha20} in the case of $k=2$ (see the entry A213275 by 
Alois P. Heinz in the OEIS for the general $k$): words $w$ of length $(k+1)n$ in an alphabet $\{a_1,\dots,a_{n}\}$ with $n$ distinct letters, such that each letter occurs $k+1$ times in the word, and for each prefix $z$ of $w$, either $a_i$ does not occur in $z$ or if it does, then for each $j>i$, it occurs more or the same number of times as $a_j$. We establish this in Theorem~\ref{ThmF}. No explicit formulas for counting these combinatorial objects seem to be known.

\item Maximal chains in iteration of stacking lattices in the third column onwards in Figure~\ref{FigA} appear to be new combinatorial objects. At least, we could not find the corresponding numbers, generated on a computer, in the OEIS (see Figure~\ref{FigB}).
\end{itemize}
All three of these results arise as applications of a `representation theorem' (Thereom~\ref{ThmD}), which gives an embedding of $k$-iterated stacking of $m$-fold hypercubes into $(k+m)$-fold hypercubes, where by an \emph{$r$-fold hypercube} we mean a lattice of the form $C^n_r$ (i.e., a lattice in the $r$-th column of Figure~~\ref{FigA}). We formulate this theorem in Section~\ref{SecB}, but defer its proof to Section~\ref{sec proof}. The representation theorem also helped with the visualization of the corresponding stacked lattices --- see Figures~\ref{Figure stacking hypercubes} and \ref{Figure stacking m=2 cubes}, where $\Sigma^k_nC^n_m$ denotes the $k$-iterated stacking of $m$-fold hypercubes up to dimension $n$. 

In Section~\ref{sec remarks}, we expand on the link with lax colimits in the $2$-category of posets, and establish some algebraic properties of this construction to get an insight as to \emph{what kinds of lattices may arise when considering stacking of lattices in Figure~\ref{FigA}?} In particular, we show that first of all, they are indeed lattices, and secondly, they are in fact distributive lattices. We do not know, however, whether distributivity of these lattices can play a role in the combinatorial investigation of their maximal chains. 
Section~\ref{sec remarks} also prepares a way to Section~\ref{sec proof}, which is devoted to the proof of the representation theorem mentioned earlier. The proof is based on decomposing lax colimits of chains of homomorphisms into another kind of (weighted) colimits in category theory, which we call `lax pushouts' (although they do not form a particular type of lax colimits). This method enables one to almost trivialize the geometric complexity of the stacked lattices of $m$-fold hypercubes. A similar method can be used to prove a representation theorem (Theorem~\ref{ThmE}) for iterated stacking of lattices along the rows of Figure~\ref{FigA}, which we formulate in the last Section~\ref{Sec rows}. As did the previous representation theorem for stacked lattices along the columns of Figure~\ref{FigA}, this representation theorem allows one to easily generate on a computer the first few terms in the corresponding sequences of maximal chain numbers, and just as before, we quickly encounter new sequences --- see Figure~\ref{FigC}.

This paper is aimed at readers of diverse background. While we bring together combinatorics, lattice theory and category theory, we took particular care in the presentation to make the paper accessible to non-experts of each of these fields (moreover, the paper uses only basic concepts from these fields). We hope that our work will entice further research on combinatorics of maximal chains of stacked lattices (or posets, more generally). Among questions left unresolved in this paper is the question of finding explicit formulas for those integer sequences from Figures~\ref{FigB} and \ref{FigC} that do not appear in OEIS, or proving that they do not exist. We anticipate these questions to be highly non-trivial.

%As suggested by the abstract, we would like to view this paper as a foundation for research on a combinatorial study of maximal chains in stacked lattices. We hope that the paper will entice the reader to attempt some of the problems left unsolved in this paper, such as finding explicit formulas for those integer sequences from Figures~\ref{FigB} and \ref{FigC} that do not appear in OEIS, or proving that they do not exist. 

\section{Stacking hypercubes}\label{SecA}

Consider the sequence $C_1^\infty=(C_1^0,C_1^1,C_1^2,C_1^3,\dots)$ of hypercubes: each $C_1^n$ denotes the $n$-th cartesian power of the $1$-chain $C_1$. Thus, each $C_1^n$ is an $n$-dimensional cube. We will stack these along the obvious embeddings $C_1^i\to C_1^{i+1}$, visualized below in the case when $i=3$:

$$
		
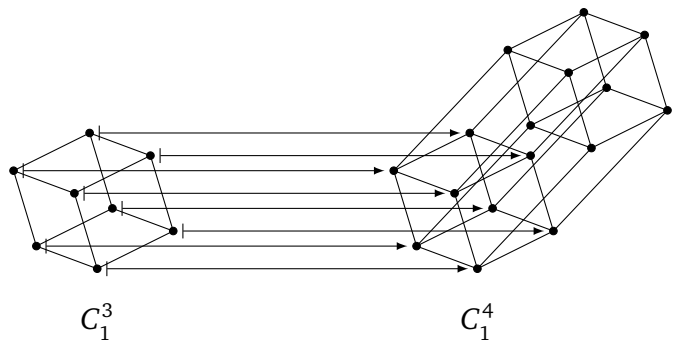
\begin{tikzpicture}
		\pgfmathsetmacro{\ix}{-0.8}
		\pgfmathsetmacro{\iy}{0.3}
		\pgfmathsetmacro{\jx}{-0.3}
		\pgfmathsetmacro{\jy}{1}
		\pgfmathsetmacro{\kx}{1.5}
		\pgfmathsetmacro{\ky}{1.6}
		\pgfmathsetmacro{\lx}{1}
		\pgfmathsetmacro{\ly}{0.5}
		
		\node at(0,-0.7) {$C_1^4$};
		\node at(-5,-0.7) {$C_1^3$};
		\tikzset{every node/.style={shape=circle,draw=black,fill=black,inner sep=1pt,outer sep=2pt}}	
		
		\foreach \i in {0,1} {
			\foreach \j in {0,1} {
				\foreach \k in {0,1} {
					\foreach \l in {0,1} {
					\node (\i\j\k\l) at(\ix*\i+\jx*\j+\kx*\k+\lx*\l,\iy*\i+\jy*\j+\ky*\k+\ly*\l){};
					\ifthenelse{\i=0}{\draw (\ix*\i+\jx*\j+\kx*\k+\lx*\l,\iy*\i+\jy*\j+\ky*\k+\ly*\l) -- (\ix*\i+\ix+\jx*\j+\kx*\k+\lx*\l,\iy*\i+\iy+\jy*\j+\ky*\k+\ly*\l);}{}
					\ifthenelse{\j=0}{\draw (\ix*\i+\jx*\j+\kx*\k+\lx*\l,\iy*\i+\jy*\j+\ky*\k+\ly*\l) -- (\ix*\i+\jx*\j+\jx+\kx*\k+\lx*\l,\iy*\i+\jy*\j+\jy+\ky*\k+\ly*\l);}{}
					\ifthenelse{\k=0}{\draw (\ix*\i+\jx*\j+\kx*\k+\lx*\l,\iy*\i+\jy*\j+\ky*\k+\ly*\l) -- (\ix*\i+\jx*\j+\kx*\k+\kx+\lx*\l,\iy*\i+\jy*\j+\ky*\k+\ky+\ly*\l);}{}
					\ifthenelse{\l=0}{\draw (\ix*\i+\jx*\j+\kx*\k+\lx*\l,\iy*\i+\jy*\j+\ky*\k+\ly*\l) -- (\ix*\i+\jx*\j+\kx*\k+\lx*\l+\lx,\iy*\i+\jy*\j+\ky*\k+\ly*\l+\ly);}{}
					}
				}
			}
		}

		\pgfmathsetmacro{\movex}{-5}
		\pgfmathsetmacro{\movey}{0}
				
		\foreach \i in {0,1} {
			\foreach \j in {0,1} {
					\foreach \l in {0,1} {
					\node (\i\j\l) at(\movex+\ix*\i+\jx*\j+\lx*\l,\movey+\iy*\i+\jy*\j+\ly*\l){};
					\ifthenelse{\i=0}{\draw (\movex+\ix*\i+\jx*\j+\lx*\l,\movey+\iy*\i+\jy*\j+\ly*\l) -- (\movex+\ix*\i+\ix+\jx*\j+\lx*\l,\movey+\iy*\i+\iy+\jy*\j+\ly*\l);}{}
					\ifthenelse{\j=0}{\draw (\movex+\ix*\i+\jx*\j+\lx*\l,\movey+\iy*\i+\jy*\j+\ly*\l) -- (\movex+\ix*\i+\jx*\j+\jx+\lx*\l,\movey+\iy*\i+\jy*\j+\jy+\ly*\l);}{}
					\ifthenelse{\l=0}{\draw (\movex+\ix*\i+\jx*\j+\lx*\l,\movey+\iy*\i+\jy*\j+\ly*\l) -- (\movex+\ix*\i+\jx*\j+\lx*\l+\lx,\movey+\iy*\i+\jy*\j+\ly*\l+\ly);}{}
					\draw [|->] (\i\j\l) -- (\i\j0\l);
					}
			}
		}
		\end{tikzpicture}$$

\noindent The numbers $\#C_1^n$ of maximal chains in the members of the sequence $C_1^\infty$ are of course given by the factorials
$$\#C_1^n=n! \quad (n\geqslant 0).$$

\begin{theorem}\label{ThmG}
The number of maximal chains in the lattice $\Sigma_n C_1^n$ obtained by stacking the hypercubes $C_1^0\to C_1^1\to \dots\to C_1^n$ is given by odd double factorials (where $(-1)!!=1$):
$$\#\Sigma_n C_1^n=(2n-1)!!$$
\end{theorem}

\begin{proof}
It is well known (see e.g.~\cite{Vie06}) that odd double factorials count the number of involutions on a set with $2n$ elements with no fixed points (which are the same as $2$-partitions). To see how, notice that such an involution can be represented as a list of distinct elements of $2n$ (permutation), where each consecutive pair of odd and neighboring even entry in the list is an involuted pair. There are $(2n)!$ such lists, but we get them
too often: for each pair, we must divide by $2$, so altogether by $2^n$, and then the order of the pairs does not matter, so we must divide by $n!$. This gives
$$\frac{(2n)!}{2^n n!}=(2n-1)!!.$$
We will now establish a bijection between maximal chains in $\Sigma_nC_1^n$ and involutions of a $2n$ element set having no fixed points. A walk along such chain visits each cube, makes some steps in the cube, and moves on to the next cube: 

$$
			
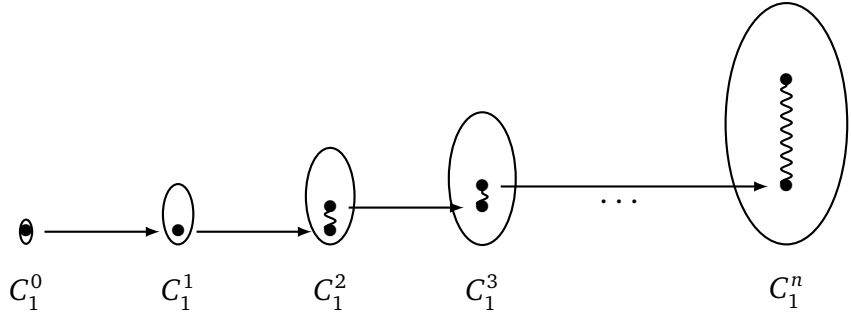
\begin{tikzpicture}[scale = 0.4, line width = 0.1mm]
			\draw[thick] (0,0+0mm) ellipse (2mm and 4mm);
			\draw[thick] (5,0+5.8mm) ellipse (5mm and 10mm);
			\draw[thick] (10,+11.8mm) ellipse (8mm and 16mm);
			\draw[thick] (15,+17.6mm) ellipse (11mm and 22mm);
			\draw[thick] (25,+35.8mm) ellipse (20mm and 40mm);

			\node (1) at (0, 0){$\bullet$};
			\node  at (0, -2){$C_1^0$};
			\node (2) at (5, 0){$\bullet$};
			\node  at (5, -2){$C_1^1$};
			\node (3) at (10, 0){$\bullet$};
			\node  at (10, -2){$C_1^2$};
			\node (3a) at (10, 0.8){$\bullet$};
			\node (4) at (15, 0.8){$\bullet$};
			\node  at (15, -2){$C_1^3$};
			\node (4a) at (15, 1.5){$\bullet$};
			\node (5) at (25, 1.5){$\bullet$};
			\node  at (25, -2){$C_1^n$};
			\node (5a) at (25, 5){$\bullet$};
			\draw [->, thick](1) -- (2);
			\draw [->, thick](2) -- (3);
			\draw [->, thick](3a) -- (4);
			\draw [->, thick](4a) -- (5);
			
			\draw[  thick, snake it](10, 0) -- (10, 0.8);
			\draw[  thick, snake it](15,0.8) -- (15,1.5);
			\draw[  thick, snake it](25,1.5) -- (25,5);
			
			\node at (17+2, 1){$.$};
			\node at (17.5+2, 1){$.$};
			\node at (18+2, 1){$.$};
			
			\end{tikzpicture}
$$

\noindent Let $i_j$ represent height reached in the cube $C^j_1$ before the exit ($i_0=0$). Then $i_j$ is also the height at which the next cube $C^{j+1}_1$ is entered. The maximum height that can be reached in $C^j_1$ is $j+1$. The number of possible paths that can be taken inside $C^{j+1}_1$ after entering it at height $i_{j}$ and before existing it at height $i_{j+1}$ is the falling factorial 
$$(j+1-i_{j})(j+1-i_{j}-1)\dots(j+1-i_{j+1}-(i_{j+1}-i_{j}-1))=(j+1-i_{j})^{\underline{i_{j+1}-i_{j}}}.$$
The number of maximal chains in $\Sigma_nC^n_1$ is thus given by
\begin{equation*}
\#\Sigma_nC^n_1=\sum_{0=i_0\leqslant i_1\leqslant \dots \leqslant  i_n=n}\prod_{j=0}^{n-1} (j+1-i_{j})^{\underline{i_{j+1}-i_{j}}}.
\end{equation*}
The indices in this summation form a Dyck path. Here is an example (with $n=7$):

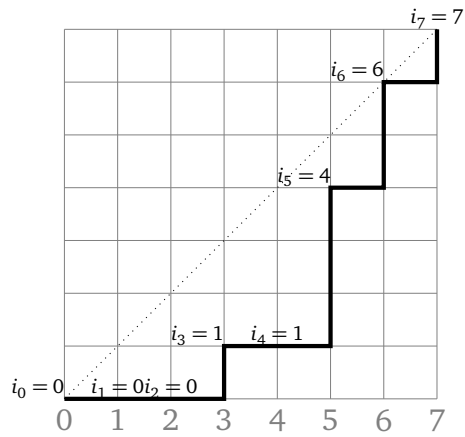
\newcommand*{\xMin}{0}
\newcommand*{\xMax}{7}
\newcommand*{\yMin}{0}
\newcommand*{\yMax}{7}
$$
		\begin{tikzpicture}[scale=0.7]
		\foreach \i in {\xMin,...,\xMax} {
			\draw [very thin,gray] (\i,\yMin) -- (\i,\yMax)  node [below] at (\i,\yMin) {$\i$};
		}
	\foreach \i in {\yMin,...,\yMax} {
		\draw [very thin,gray] (\xMin,\i) -- (\xMax,\i) node [left] at (\xMin,\i) { };
		}
		
		\draw[dotted] (0,0) -- (7,7);
		\draw[ultra thick] (0,0) -- (3,0)--(3,1)--(5,1)--(5,4)--(6,4)--(6,6)--(7,6)--(7,7);
		
		\node  at(-0.5,0.2){\tiny $i_0=0$};
		\node  at(1,0.2){\tiny $i_1=0$};
		\node  at(2,0.2){\tiny $i_2=0$};
		
		\node  at(2.5,1.2){\tiny $i_3=1$};
			\node  at(4,1.2){\tiny $i_4=1$};
			\node  at(4.5,4.2){\tiny $i_5=4$};
			\node  at(5.5,6.2){\tiny $i_6=6$};
			\node  at(7,7.2){\tiny $i_7=7$};
		\end{tikzpicture}
$$

\noindent This is a `weighted' Dyck path, with the weight given by the terms from the above summation:
 \begin{equation*}
1^{\underline{0}}\;2^{\underline{0}}\;3^{\underline{1}}\;3^{\underline{0}}\;4^{\underline{3}}\;2^{\underline{2}}\;1^{\underline{1}}
=3\cdot4\cdot3\cdot2\cdot2\cdot1\cdot1.
 \end{equation*}
Here is a drawing with the weights distributed on the Dyck path:

$$
		
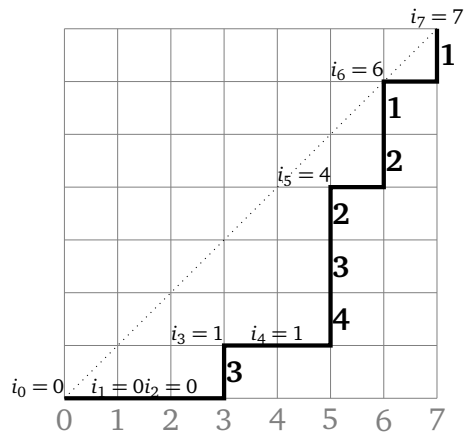
\begin{tikzpicture}[scale=0.7]
		\foreach \i in {\xMin,...,\xMax} {
			\draw [very thin,gray] (\i,\yMin) -- (\i,\yMax)  node [below] at (\i,\yMin) {$\i$};
		}
		\foreach \i in {\yMin,...,\yMax} {
			\draw [very thin,gray] (\xMin,\i) -- (\xMax,\i) node [left] at (\xMin,\i) { };
		}
		
		\draw[dotted] (0,0) -- (7,7);
		\draw[ultra thick] (0,0) -- (3,0)--(3,1)--(5,1)--(5,4)--(6,4)--(6,6)--(7,6)--(7,7);
		
		\node  at(-0.5,0.2){\tiny $i_0=0$};
		\node  at(1,0.2){\tiny $i_1=0$};
		\node  at(2,0.2){\tiny $i_2=0$};
		
		\node  at(2.5,1.2){\tiny $i_3=1$};
		\node  at(4,1.2){\tiny $i_4=1$};
		\node  at(4.5,4.2){\tiny $i_5=4$};
		\node  at(5.5,6.2){\tiny $i_6=6$};
		\node  at(7,7.2){\tiny $i_7=7$};
		
		\node  at(7.2,6.5){$\mathbf{1}$};
\node  at(6.2,5.5){$\mathbf{1}$};
\node  at(6.2,4.5){$\mathbf{2}$};		
\node  at(5.2,3.5){$\mathbf{2}$};		
\node  at(5.2,2.5){$\mathbf{3}$};		
\node  at(5.2,1.5){$\mathbf{4}$};
\node  at(3.2,0.5){$\mathbf{3}$};				

		\end{tikzpicture}
$$

\noindent Each bold-face number in the display above represents number of choices when going upward in the given cube. A more traditional way of drawing this is:

$$
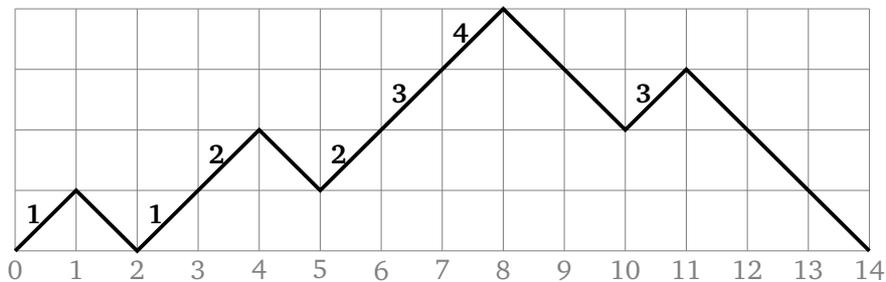
\resizebox{0.7\textwidth}{!}{
		\begin{tikzpicture}[scale=0.9]
		\foreach \i in {\xMin,...,14} {
			\draw [very thin,gray] (\i,0) -- (\i,4)  node [below] at (\i,\yMin) {$\i$};
		}
		\foreach \i in {\yMin,...,4} {
			\draw [very thin,gray] (\xMin,\i) -- (2*\xMax,\i) node [left] at (\xMin,\i) { };
		}

		\draw[ultra thick] (0,0) -- (1,1)--(2,0)--(4,2)--(5,1)--(8,4)--(10,2)--(11,3)--(14,0);

		\node  at(0.3,0.6){$\mathbf{1}$};
		\node  at(2.3,0.6){$\mathbf{1}$};
		\node  at(3.3,1.6){$\mathbf{2}$};		
		\node  at(5.3,1.6){$\mathbf{2}$};		
		\node  at(6.3,2.6){$\mathbf{3}$};		
		\node  at(7.3,3.6){$\mathbf{4}$};
		\node  at(10.3,2.6){$\mathbf{3}$};				
		
		\end{tikzpicture}}
$$

\noindent Notice that the weights at each level match with the height of the level.
In his video book, X.~Viennot ($n!$--garden, part (b)) constructs a fixed-point free involution on the set $\{1,2,\dots,14\}$ from such a weighted Dyck path (which he calls `Hermite history'). First, he draws the following, which indicates whether at each position the Dyck path makes an `up-step' or a `down-step'.  

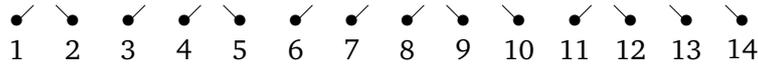
$$
\resizebox{0.6\textwidth}{!}{		\begin{tikzpicture}[scale=0.8]
		\foreach \i in {1,...,14} {
			 \node at(\i,0){$\bullet$};
		}
		\foreach \i in {1,...,14} {
			\node at(\i,-0.5){$\i$};
		}
	\foreach \i in {1,3,4,6,7,8,11} {
		\draw (\i,0.0) -- (\i+0.3,0.3);
	}
\foreach \i in {2,5,9,10,12,13,14} {
	\draw (\i,0.0) -- (\i-0.3,0.3);
}
	
		\end{tikzpicture}}
$$

\noindent He will pair each up-step-node with a down-step-node, but which one? We have choices! Copy the weights from the Dyck path:

$$
	\resizebox{0.6\textwidth}{!}
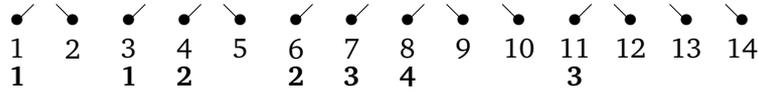
{	\begin{tikzpicture}[scale=0.8]
		\foreach \i in {1,...,14} {
			\node at(\i,0){$\bullet$};
		}
		\foreach \i in {1,...,14} {
			\node at(\i,-0.5){$\i$};
		}
		\foreach \i in {1,3,4,6,7,8,11} {
			\draw (\i,0.0) -- (\i+0.3,0.3);
					}
		
		\node at (1,-1){$\mathbf{1}$};
		\node at (3,-1){$\mathbf{1}$};
		\node at (4,-1){$\mathbf{2}$};
		\node at (6,-1){$\mathbf{2}$};
		\node at (7,-1){$\mathbf{3}$};
		\node at (8,-1){$\mathbf{4}$};
		\node at (11,-1){$\mathbf{3}$};		
				
		\foreach \i in {2,5,9,10,12,13,14} {
			\draw (\i,0.0) -- (\i-0.3,0.3);
		}
		
		\end{tikzpicture}}
$$

\noindent As we can see, one may connect $11$ with one of $12,13,14$. The boldface number $\mathbf 3$ now tells us that we have 3 options for such a pairing. Write below the bold-face numbers which options will be selected:

$$
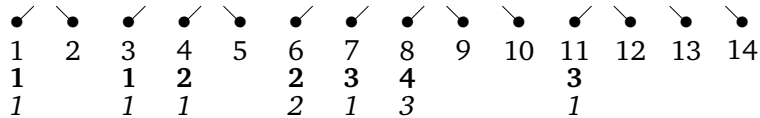
\resizebox{0.6\textwidth}{!}{		\begin{tikzpicture}[scale=0.8]
		\foreach \i in {1,...,14} {
			\node at(\i,0){$\bullet$};
		}
		\foreach \i in {1,...,14} {
			\node at(\i,-0.5){$\i$};
		}
		\foreach \i in {1,3,4,6,7,8,11} {
			\draw (\i,0.0) -- (\i+0.3,0.3);
		}
		
		\node at (1,-1){$\mathbf{1}$};
		\node at (3,-1){$\mathbf{1}$};
		\node at (4,-1){$\mathbf{2}$};
		\node at (6,-1){$\mathbf{2}$};
		\node at (7,-1){$\mathbf{3}$};
		\node at (8,-1){$\mathbf{4}$};
		\node at (11,-1){$\mathbf{3}$};

		\node at (1,-1.5){$\mathit{1}$};
		\node at (3,-1.5){$\mathit{1}$};
		\node at (4,-1.5){$\mathit{1}$};
		\node at (6,-1.5){$\mathit{2}$};
		\node at (7,-1.5){$\mathit{1}$};
		\node at (8,-1.5){$\mathit{3}$};
		\node at (11,-1.5){$\mathit{1}$};

		\foreach \i in {2,5,9,10,12,13,14} {
			\draw (\i,0.0) -- (\i-0.3,0.3);
		}
		
		\end{tikzpicture}}
$$

\noindent Now we work off the up-step-nodes from right to left, and connect with our choice. For 11 we choose 1, therefore our first choice (from the right), which is node $14$:

$$
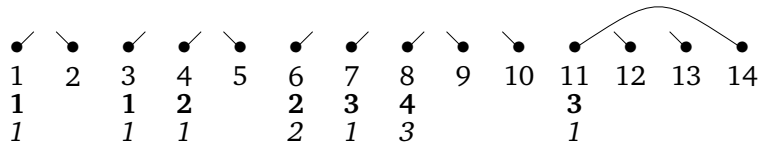
\resizebox{0.6\textwidth}{!}{		\begin{tikzpicture}[scale=0.8]
		\foreach \i in {1,...,14} {
			\node at(\i,0){$\bullet$};
		}
		\foreach \i in {1,...,14} {
			\node at(\i,-0.5){$\i$};
		}
		\foreach \i in {1,3,4,6,7,8} {
			\draw (\i,0.0) -- (\i+0.3,0.3);
		}
			\foreach \i in {2,5,9,10,12,13} {
			\draw (\i,0.0) -- (\i-0.3,0.3);
}		
			
		\node at (1,-1){$\mathbf{1}$};
		\node at (3,-1){$\mathbf{1}$};
		\node at (4,-1){$\mathbf{2}$};
		\node at (6,-1){$\mathbf{2}$};
		\node at (7,-1){$\mathbf{3}$};
		\node at (8,-1){$\mathbf{4}$};
		\node at (11,-1){$\mathbf{3}$};

		\node at (1,-1.5){$\mathit{1}$};
		\node at (3,-1.5){$\mathit{1}$};
		\node at (4,-1.5){$\mathit{1}$};
		\node at (6,-1.5){$\mathit{2}$};
		\node at (7,-1.5){$\mathit{1}$};
		\node at (8,-1.5){$\mathit{3}$};
		\node at (11,-1.5){$\mathit{1}$};
		\draw (11,0) .. controls (12.5,1) .. (14,0);

		\end{tikzpicture}}
$$

\noindent We then move to the next node, which is labeled 8, and connect it with the $3$rd option (from the right), which is node $10$:

$$
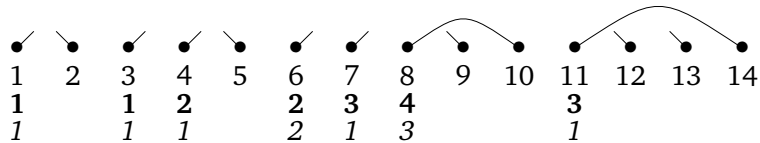
\resizebox{0.6\textwidth}{!}{		\begin{tikzpicture}[scale=0.8]
		\foreach \i in {1,...,14} {
			\node at(\i,0){$\bullet$};
		}
		\foreach \i in {1,...,14} {
			\node at(\i,-0.5){$\i$};
		}
		\foreach \i in {1,3,4,6,7} {
			\draw (\i,0.0) -- (\i+0.3,0.3);
		}
		\foreach \i in {2,5,9,12,13} {
			\draw (\i,0.0) -- (\i-0.3,0.3);
}			
			
			\node at (1,-1){$\mathbf{1}$};
			\node at (3,-1){$\mathbf{1}$};
			\node at (4,-1){$\mathbf{2}$};
			\node at (6,-1){$\mathbf{2}$};
			\node at (7,-1){$\mathbf{3}$};
			\node at (8,-1){$\mathbf{4}$};
			\node at (11,-1){$\mathbf{3}$};

			\node at (1,-1.5){$\mathit{1}$};
			\node at (3,-1.5){$\mathit{1}$};
			\node at (4,-1.5){$\mathit{1}$};
			\node at (6,-1.5){$\mathit{2}$};
			\node at (7,-1.5){$\mathit{1}$};
			\node at (8,-1.5){$\mathit{3}$};
			\node at (11,-1.5){$\mathit{1}$};
			\draw (11,0) .. controls (12.5,1) .. (14,0);
			\draw (8,0) .. controls (9,0.7) .. (10,0);

		\end{tikzpicture}}
$$

\noindent Now we move to the next node $7$ and connect it with the first option (from the right), which is node $13$:

$$
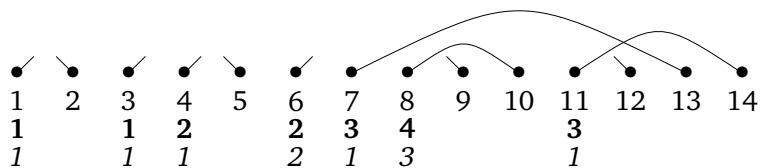
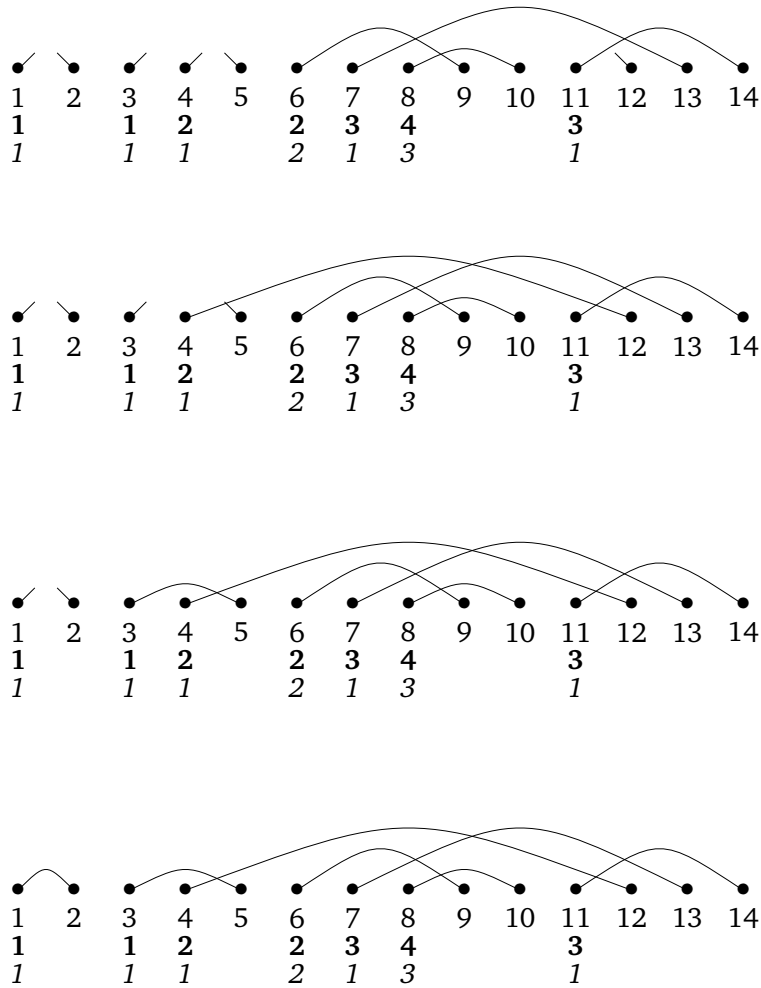
\resizebox{0.6\textwidth}{!}{		\begin{tikzpicture}[scale=0.8]
		\foreach \i in {1,...,14} {
			\node at(\i,0){$\bullet$};
		}
		\foreach \i in {1,...,14} {
			\node at(\i,-0.5){$\i$};
		}
		\foreach \i in {1,3,4,6} {
			\draw (\i,0.0) -- (\i+0.3,0.3);
		}
		\foreach \i in {2,5,9,12} {
			\draw (\i,0.0) -- (\i-0.3,0.3);
	}	
			
			\node at (1,-1){$\mathbf{1}$};
			\node at (3,-1){$\mathbf{1}$};
			\node at (4,-1){$\mathbf{2}$};
			\node at (6,-1){$\mathbf{2}$};
			\node at (7,-1){$\mathbf{3}$};
			\node at (8,-1){$\mathbf{4}$};
			\node at (11,-1){$\mathbf{3}$};

			\node at (1,-1.5){$\mathit{1}$};
			\node at (3,-1.5){$\mathit{1}$};
			\node at (4,-1.5){$\mathit{1}$};
			\node at (6,-1.5){$\mathit{2}$};
			\node at (7,-1.5){$\mathit{1}$};
			\node at (8,-1.5){$\mathit{3}$};
			\node at (11,-1.5){$\mathit{1}$};
			\draw (11,0) .. controls (12.5,1) .. (14,0);
			\draw (8,0) .. controls (9,0.7) .. (10,0);
			\draw (7,0) .. controls (10,1.5) .. (13,0);

		\end{tikzpicture}}
$$

\noindent And so forth:

$$
\resizebox{0.6\textwidth}{!}{		\begin{tikzpicture}[scale=0.8]
		\foreach \i in {1,...,14} {
			\node at(\i,0){$\bullet$};
		}
		\foreach \i in {1,...,14} {
			\node at(\i,-0.5){$\i$};
		}
		\foreach \i in {1,3,4} {
			\draw (\i,0.0) -- (\i+0.3,0.3);
		}
		\foreach \i in {2,5,12} {
			\draw (\i,0.0) -- (\i-0.3,0.3);
	}	
			
			\node at (1,-1){$\mathbf{1}$};
			\node at (3,-1){$\mathbf{1}$};
			\node at (4,-1){$\mathbf{2}$};
			\node at (6,-1){$\mathbf{2}$};
			\node at (7,-1){$\mathbf{3}$};
			\node at (8,-1){$\mathbf{4}$};
			\node at (11,-1){$\mathbf{3}$};

			\node at (1,-1.5){$\mathit{1}$};
			\node at (3,-1.5){$\mathit{1}$};
			\node at (4,-1.5){$\mathit{1}$};
			\node at (6,-1.5){$\mathit{2}$};
			\node at (7,-1.5){$\mathit{1}$};
			\node at (8,-1.5){$\mathit{3}$};
			\node at (11,-1.5){$\mathit{1}$};
			\draw (11,0) .. controls (12.5,1) .. (14,0);
			\draw (8,0) .. controls (9,0.5) .. (10,0);
			\draw (7,0) .. controls (10,1.5) .. (13,0);
			\draw (6,0) .. controls (7.5,1) .. (9,0);

		\end{tikzpicture}}
$$
$$
\resizebox{0.6\textwidth}{!}{		\begin{tikzpicture}[scale=0.8]
		\foreach \i in {1,...,14} {
			\node at(\i,0){$\bullet$};
		}
		\foreach \i in {1,...,14} {
			\node at(\i,-0.5){$\i$};
		}
		\foreach \i in {1,3} {
			\draw (\i,0.0) -- (\i+0.3,0.3);
		}
		\foreach \i in {2,5} {
			\draw (\i,0.0) -- (\i-0.3,0.3);
}		
			
			\node at (1,-1){$\mathbf{1}$};
			\node at (3,-1){$\mathbf{1}$};
			\node at (4,-1){$\mathbf{2}$};
			\node at (6,-1){$\mathbf{2}$};
			\node at (7,-1){$\mathbf{3}$};
			\node at (8,-1){$\mathbf{4}$};
			\node at (11,-1){$\mathbf{3}$};

			\node at (1,-1.5){$\mathit{1}$};
			\node at (3,-1.5){$\mathit{1}$};
			\node at (4,-1.5){$\mathit{1}$};
			\node at (6,-1.5){$\mathit{2}$};
			\node at (7,-1.5){$\mathit{1}$};
			\node at (8,-1.5){$\mathit{3}$};
			\node at (11,-1.5){$\mathit{1}$};
			\draw (11,0) .. controls (12.5,1) .. (14,0);
			\draw (8,0) .. controls (9,0.5) .. (10,0);
			\draw (7,0) .. controls (10,1.5) .. (13,0);
			\draw (6,0) .. controls (7.5,1) .. (9,0);
			\draw (4,0) .. controls (8,1.5) .. (12,0);

		\end{tikzpicture}}
$$

$$
\resizebox{0.6\textwidth}{!}{		\begin{tikzpicture}[scale=0.8]
		\foreach \i in {1,...,14} {
			\node at(\i,0){$\bullet$};
		}
		\foreach \i in {1,...,14} {
			\node at(\i,-0.5){$\i$};
		}
		\foreach \i in {1} {
			\draw (\i,0.0) -- (\i+0.3,0.3);
		}
		\foreach \i in {2} {
			\draw (\i,0.0) -- (\i-0.3,0.3);
	}	
			
			\node at (1,-1){$\mathbf{1}$};
			\node at (3,-1){$\mathbf{1}$};
			\node at (4,-1){$\mathbf{2}$};
			\node at (6,-1){$\mathbf{2}$};
			\node at (7,-1){$\mathbf{3}$};
			\node at (8,-1){$\mathbf{4}$};
			\node at (11,-1){$\mathbf{3}$};

			\node at (1,-1.5){$\mathit{1}$};
			\node at (3,-1.5){$\mathit{1}$};
			\node at (4,-1.5){$\mathit{1}$};
			\node at (6,-1.5){$\mathit{2}$};
			\node at (7,-1.5){$\mathit{1}$};
			\node at (8,-1.5){$\mathit{3}$};
			\node at (11,-1.5){$\mathit{1}$};
			\draw (11,0) .. controls (12.5,1) .. (14,0);
			\draw (8,0) .. controls (9,0.5) .. (10,0);
			\draw (7,0) .. controls (10,1.5) .. (13,0);
			\draw (6,0) .. controls (7.5,1) .. (9,0);
			\draw (4,0) .. controls (8,1.5) .. (12,0);
			\draw (3,0) .. controls (4,0.5) .. (5,0);

		\end{tikzpicture}}
$$

$$
\resizebox{0.6\textwidth}{!}{		\begin{tikzpicture}[scale=0.8]
		\foreach \i in {1,...,14} {
			\node at(\i,0){$\bullet$};
		}
		\foreach \i in {1,...,14} {
			\node at(\i,-0.5){$\i$};
		}

			\node at (1,-1){$\mathbf{1}$};
			\node at (3,-1){$\mathbf{1}$};
			\node at (4,-1){$\mathbf{2}$};
			\node at (6,-1){$\mathbf{2}$};
			\node at (7,-1){$\mathbf{3}$};
			\node at (8,-1){$\mathbf{4}$};
			\node at (11,-1){$\mathbf{3}$};

			\node at (1,-1.5){$\mathit{1}$};
			\node at (3,-1.5){$\mathit{1}$};
			\node at (4,-1.5){$\mathit{1}$};
			\node at (6,-1.5){$\mathit{2}$};
			\node at (7,-1.5){$\mathit{1}$};
			\node at (8,-1.5){$\mathit{3}$};
			\node at (11,-1.5){$\mathit{1}$};
			\draw (11,0) .. controls (12.5,1) .. (14,0);
			\draw (8,0) .. controls (9,0.5) .. (10,0);
			\draw (7,0) .. controls (10,1.5) .. (13,0);
			\draw (6,0) .. controls (7.5,1) .. (9,0);
			\draw (4,0) .. controls (8,1.5) .. (12,0);
			\draw (3,0) .. controls (4,0.5) .. (5,0);
			\draw (1,0) .. controls (1.5,0.5) .. (2,0);

		\end{tikzpicture}}
$$

\noindent The goal is achieved; we constructed an involution.\end{proof}

To see the bijection between $2$-partitions of a $2n$-element set and maximal chains in the lattice $\Sigma_nC^n_1$ more directly, we first embed $\Sigma_nC^n_1$ in $C_2^{n}$ as follows. Represent elements of $C_2$ as $0$, $1$ and $2$, in the increasing order. Then elements of $C_2^{n}$ can be represented as strings of $0,1,2$ of length $n$. On the other hand, if we write $1$ and $2$ for the elements of $C_1$, then elements of each $C_1^j$ can be represented as strings of $1,2$ of length $j$. Embed $C_1^j$ into $C_2^{n}$ by adding in front $0$'s to fill up each string of length $j$ to a string of length $n$. A walk along a maximal chain in $\Sigma_nC^n_1$ now becomes a walk along a maximal chain in $C_2^{n}$ which passes through only those strings where a zero never follows a nonzero entry. Each time we make a step along such walk, we have two choices: either to increment any of the nonzero entries in the string, or to increment the right-most $0$ entry to $1$. This second choice corresponds to moving to the next cube, whereas the first one, moving up in the same cube. There are altogether $2n$ steps to make. Pair each $j$-th step of the first type to the $j'$-th step of the second type, at both steps $i$-th entry was incremented. Note that each step can only be incremented twice and the step of the first type incrementing $i$-th entry will always succeed step of the second type incrementing the $i$-th entry (first I would have to move to the cube that has the dimension along which I want to make a step, before such step can be made). The $2$-partition displayed above, in the case when $n=14$, will then give rise to the following walk:

$$
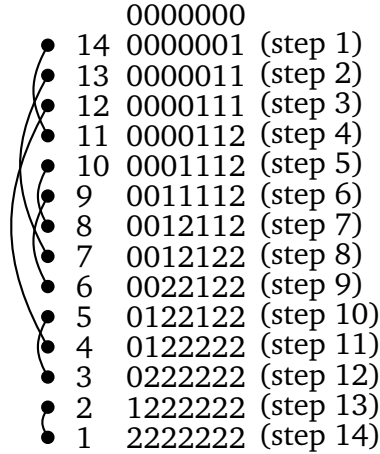
\begin{tikzpicture}[->,>=stealth',thick,scale=0.8]
	\node at(0,0){0000000};
	\node at(0,-0.5){0000001};
	\node at(0,-1){0000011};
	\node at(0,-1.5){0000111};
	\node at(0,-2){0000112};
	\node at(0,-2.5){0001112};
	\node at(0,-3){0011112};
	\node at(0,-3.5){0012112};
	\node at(0,-4){0012122};
	\node at(0,-4.5){0022122};
	\node at(0,-5){0122122};
	\node at(0,-5.5){0122222};
	\node at(0,-6){0222222};
	\node at(0,-6.5){1222222};
	\node at(0,-7){2222222};
	
	\foreach \i in {1,...,14}{
		\node [anchor=west] at(-2,-7.5+\i*0.5){\i};
		\node [anchor=west] at(1,0-\i*0.5){\textrm{(step \i)}};
		\tikzset{every node/.style={shape=circle,draw=black,fill=black,inner sep=1.5pt,outer sep=3pt}}
		\node [anchor=west] (\i) at(-2.5,-7.5+\i*0.5){};

	}
	
	\path [-,shorten >= -7pt, shorten <= -7pt] (11) edge [bend left] (14);
	\path [-,shorten >= -7pt, shorten <= -7pt] (8) edge [bend left] (10);
	\path [-,shorten >= -7pt, shorten <= -7pt] (7) edge [bend left] (13);
	\path [-,shorten >= -7pt, shorten <= -7pt] (6) edge [bend left] (9);
	\path [-,shorten >= -7pt, shorten <= -7pt] (4) edge [bend left] (12);
	\path [-,shorten >= -7pt, shorten <= -7pt] (3) edge [bend left] (5);
	\path [-,shorten >= -7pt, shorten <= -7pt] (1) edge [bend left] (2);
\end{tikzpicture}
$$ 

\noindent This argument easily generalizes, by replacing $C_1$ with $C_{m-1}$, to get a bijection between maximal chains in $\Sigma_nC^n_{m-1}$ and $m$-partitions of a set with $mn$ elements. We therefore get:

\begin{theorem}\label{ThmA}
For any natural number $m\geqslant 1$, there is a bijection between the set of maximal chains in the lattice $\Sigma_n C_{m-1}^n$ obtained by stacking $C_{m-1}^0\to C_{m-1}^1\to \dots\to C_{m-1}^n$, and the set of $m$-partitions of a set with $mn$ elements. Therefore,
$$\#\Sigma_n C_{m-1}^n=\frac{(mn)!}{(m!)^n n!}.$$ 
\end{theorem}

\section{Iterated stacking}\label{SecB}

The process of stacking lattices in a sequence can be iterated: the stacked lattices produce a sequence of lattices, whose members can be stacked. We write $\Sigma_n^kL_n$ for the result of $k$-th iteration, with $\Sigma_n^1L_n=\Sigma_nL_n$ and $\Sigma_n^0L_n=L_n$. Each homomorphism $\Sigma_n^kL_n\to \Sigma_{n+1}^kL_{n+1}$ used for the next iteration is given by the universal property of lax colimit and the homomorphism $\Sigma_n^{k-1}L_n\to~\Sigma_{n+1}^{k-1}L_{n+1}$ from the previous iteration (see Section~\ref{sec remarks}). The sequence $C_1,C_2,\dots$ of chains, with the inclusions $C_i\to C_{i+1}$ we have been considering, can be obtained by stacking copies of the trivial chain $C_0$ along the identity maps $C_0\to C_0\to\dots \to C_0$. Thus,
$$\Sigma_n C_0=C_n.$$ 
We then get
$$\#\Sigma^2_n C_0=\#\Sigma_n C_n=\frac{(2n)!}{n!(n+1)!}.$$
A natural question arises: what happens if we go higher in iteration? Geometric inspection of $\Sigma^k_n C_0$ shows that it is isomorphic to the portion of $C_n^k$ consisting of points with non-increasing coordinates; that is, all points whose coordinates satisfy
\[x_1\leqslant x_2\leqslant\dots\leqslant x_n.\] Hence maximal chains in $\Sigma^k_n C_0$ are counted by $$\#\Sigma^k_n C_0=\frac{(k-1)!(kn)!}{n!(n+1)!\dots(n+k-1)!},$$
nothing other than $k$-dimensional Catalan numbers \cite{SnoTro89}. 

So then, what do we get if we iterate stacking of hypercubes? Adopting ideas from the discussion before Theorem~\ref{ThmA}, we can represent each $\Sigma^k_n C^n_m$ as a subposet of $C^n_{k+m}$, consisting of those elements whose coordinates $(x_1,\dots,x_n)$ satisfy the following condition: 
\begin{itemize}
\item[($\ast$)] for each $i\in\{1,\dots,n-1\}$, if $x_{i+1}\in\{0,\dots,k-1\}$ then $x_i\leqslant x_{i+1}$.
\end{itemize} 
This representation turns out to be an embedding of lattices:

\begin{theorem}[representation of $k$-iterated stacking of $m$-fold hypercubes]\label{ThmD}
The poset $\Sigma^k_n C^n_m$ is isomorphic to the sublattice of $C^n_{k+m}$ consisting of those elements that satisfy ($\ast$).  
\end{theorem}

The formal proof of this theorem is given in Section~\ref{sec proof}.

Note that the Dyck situation and the hypercube one come together with $\Sigma^k_n C^n_m$: we get the first by letting $m=0$ and $k=2$, and the second by letting $k=1$ and $m=1$. The representation from the theorem above has been used to create drawings of stacked lattices in Figures~\ref{Figure stacking trivial lattices},~\ref{Figure stacking hypercubes} and~\ref{Figure stacking m=2 cubes}. In each row of each figure, the hollow vertices represent points mapped from the previous term.

The case $k=0$ is trivial. This is when no stacking is taking place. So in this case,
$$\#\Sigma^0_n C^n_m=\frac{(mn)!}{(m!)^n}$$
is the well-known number of maximal chains of $C^n_m$. The proof is a straightforward.  
Walking up in the lattice $C^n_m$ requires
$mn$ many steps. Write the steps out in a sequence where the first $m$ many steps are from walk in the first dimension, the second $m$ many steps in the second dimension, etc. In each group, the order of steps is insignificant, so for each group, divide $(mn)!$ by all possible permutations of the group -- that is, by $m!$. There are $n$ many such divisions that need to take place.

\begin{figure}
\input{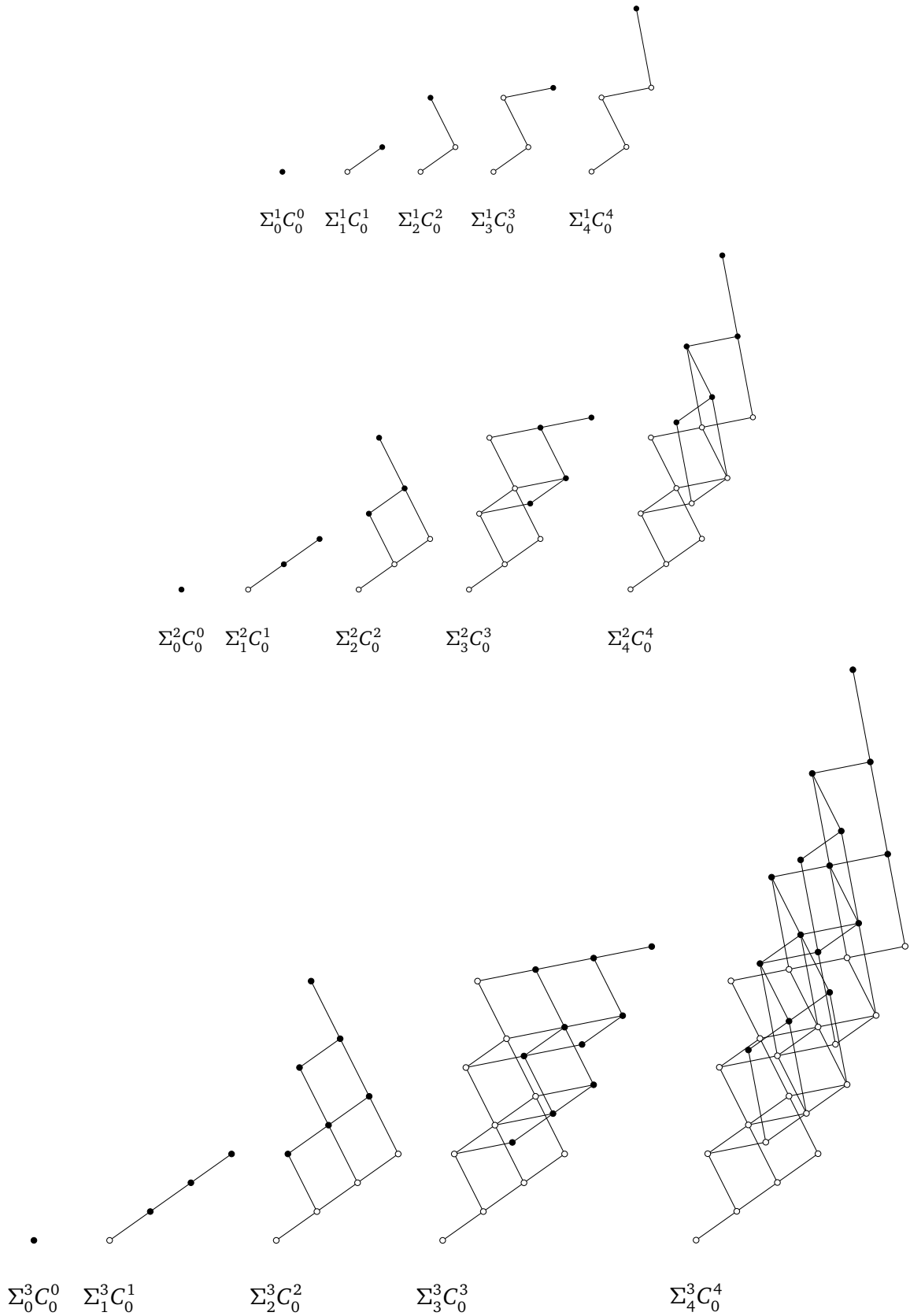}
\caption{First, second and third iteration of stacking the trivial lattice --- in the second and third iteration, maximal chains in these lattices are given by Catalan numbers and $3$-dimensional Catalan numbers, respectively.}\label{Figure stacking trivial lattices}
\end{figure}

\begin{figure}
\input{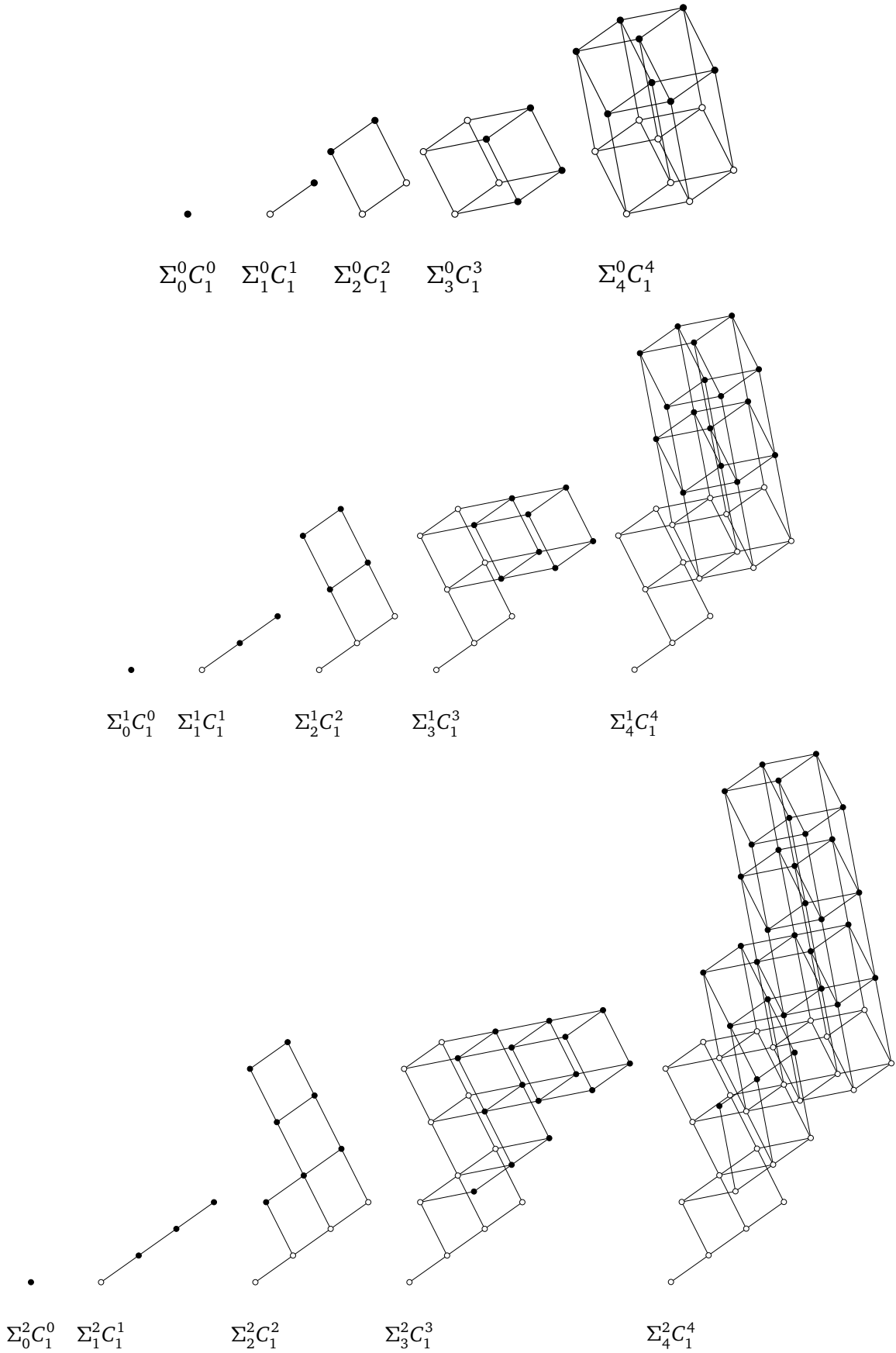}
\caption{Stacking hypercubes --- zeroth, first and second iteration.}\label{Figure stacking hypercubes}
\end{figure}

\begin{figure}
\input{hypercubes3}
\caption{
Lattices $\Sigma^k_nC^n_2$ for $k\in\{0,1,2\}$ and $n\in\{0,1,2,3,4\}$.}\label{Figure stacking m=2 cubes}
\end{figure}

\begin{figure}
$$\begin{array}{rl}
& n=0,1,2,3,4,5\dots\\\hline
m=0: \\\hline
k=0 & 1, 1, 1, 1, 1, 1, \dots \textrm{(A000012)}\\
k=1 & 1, 1, 1, 1, 1, 1, \dots \textrm{(A000012)}\\
k=2 & 1, 1, 2, 5, 14, 42, \dots \textrm{(A000108 -- Catalan numbers)}\\
k=3 & 1, 1, 5, 42, 462, 6006, \dots \textrm{(A005789 -- $3$-dimensional Catalan numbers)}\\
k=4 & 1, 1, 14, 462, 24024, 1662804, \dots \textrm{(A005790 -- $4$-dimensional Catalan numbers)}\\\hline
m=1:\\\hline
k=0 & 1, 1, 2, 6, 24, 120, \dots \textrm{(A000142 -- factorial numbers)}\\
k=1 & 1, 1, 3, 15, 105, 945, \dots \textrm{(A001147 -- double factorial numbers)}\\
k=2 & 1, 1, 7, 106, 2575, 87595, \dots \textrm{(A213863 -- no explicit formula)}\\
k=3 & 1, 1, 19, 1075, 115955, 19558470, \dots \textrm{(A213864 -- no explicit formula)}\\
k=4 & 1, 1, 56, 13326, 7364321, 7236515981, \dots \textrm{(A213865 -- no explicit formula)}\\\hline
m=2:\\\hline
k=0 & 1, 1, 6, 90, 2520, 113400, \dots \textrm{(A000680 -- $(2n)!/2^n$)}\\
k=1 & 1, 1, 10, 280, 15400, 1401400, \dots \textrm{(A025035 -- $(3n)!/(3!)^n n!$)}\\
k=2 & 1, 1, 25, 2305, 482825, 183500625, \dots \textrm{(no entry)}\\
k=3 & 1, 1, 71, 25911, 25754021, 52213860026, \dots \textrm{(no entry)}\\
k=4 & 1, 1, 216, 345651, 1848745731, 23070700145026, \dots \textrm{(no entry)}\\\hline
m=3:\\\hline
k=0 & 1, 1, 20, 1680, 369600, 168168000, \dots \textrm{(A014606 -- $(3n)!/(3!)^n$)}\\
k=1 & 1, 1, 35, 5775, 2627625, 2546168625, \dots \textrm{(A025036 -- $(4n)!/(4!)^n n!$)}\\
k=2 & 1, 1, 91, 51821, 94597041, 404793761526, \dots \textrm{(no entry)}\\
k=3 & 1, 1, 266, 621831, 5616763761, 134269580611026, \dots \textrm{(no entry)}\\
k=4 & 1, 1, 827, 8721245, 438307511209, 66953592509190248, \dots \textrm{(no entry)}\\\end{array}$$
\caption{The maximal chain numbers $\#\Sigma^k_n C^n_m$ and the corresponding entry code (indicated in brackets) in the On-Line Encyclopedia of Integer Sequences (accessed 9/08/2020).}\label{FigB}
\end{figure}

It appears that except those cases discussed in this paper, no explicit formulas are known for $\Sigma^k_n C^n_m$, and moreover, when $k>1$ and $m>1$, these combinatorial objects have not been considered in the literature. Figure~\ref{FigB} gives some of the numbers $\#\Sigma^k_n C^n_m$ generated by a computer. As indicated there (as well as in the Introduction), the numbers corresponding to the iterated stacking of hypercubes, i.e., the case when $m=1$, do show up on the On-Line Encyclopedia of Integer Sequences with an interpretation given in the following theorem. In the case when $k=2$, these combinatorial objects are in bijection with certain tree-child networks, as explained in \cite{FucYuZha20}. 

\begin{theorem}\label{ThmF}
There is a bijection between maximal chains in the lattice $\Sigma^k_n C^n_1$ and words $w$ of length $(k+1)n$ in an alphabet $\{a_1,\dots,a_{n}\}$ with $n$ distinct letters, such that each letter occurs $k+1$ times in the word, and for each prefix $z$ of $w$, either $a_i$ does not occur in $z$ or if it does, then for each $j>i$, it occurs more or the same number of times as $a_j$.   
\end{theorem} 

\begin{proof}
Let $w$ be such a word. As we read letters encountered in $w$ from left to right, record the letters that remain unused as an $n$-tuple $(x_1,\dots,x_{n})$, where $x_i$ is the number of remaining occurrences of $a_i$ in the rest of the word. The following example (where $n=2$ and $k=2$) illustrates this, where letters left to the dot are those that have been read.
$$
\begin{array}{r|l}
.a_2a_1a_1a_2a_1a_2 & (3,3)\\
a_2.a_1a_1a_2a_1a_2 & (3,2)\\
a_2a_1.a_1a_2a_1a_2 & (2,2)\\
a_2a_1a_1.a_2a_1a_2 & (1,2)\\
a_2a_1a_1a_2.a_1a_2 & (1,1)\\ 
a_2a_1a_1a_2a_1.a_2 & (0,1)\\
a_2a_1a_1a_2a_1a_2. & (0,0)\\
\end{array}
$$
With every next reading, exactly one of the terms in $(x_1,\dots,x_{n})$ decrements. We get a bijection between all words $w$ of length $(k+1)n$ in the alphabet $\{a_1,\dots,a_{n}\}$ and maximal chains in the lattice $C^n_{k+1}$ --- the points of the chain, as we descend down the chain, are given by the sequence of $n$-tuples $(x_1,\dots,x_n)$. If we could show that the condition ($\ast$) on $(x_1,\dots,x_{n})$ is equivalent to the requirement on the prefix $z$ read (which resulted in the tuple $(x_1,\dots,x_{n})$) given in the theorem, then we would be done. However, as we will now see, this equivalence actually fails. First, note the following: 
\begin{itemize}
\item $x_i=k+1$ if and only if $a_i$ does not occur in $z$.

\item $x_{i}\leqslant x_{j}$ if and only if the number of occurrences of $a_{i}$ in $z$ is more or the same as the number of occurrences of $a_{j}$.
\end{itemize}
So the requirement in the theorem on a prefix $z$ translates to the following requirement on the corresponding tuple $(x_1,\dots,x_{n})$. 
\begin{itemize}
\item[($\ast'$)] For each $i\in\{1,\dots,n\}$, either $x_i=k+1$ or $x_i\leqslant x_j$ for each $j>i$.
\end{itemize}
To see that ($\ast$) is not equivalent to ($\ast'$), consider the tuple $(3,0)$ (for $n=2$ and $k=2$). For this tuple, ($\ast'$) holds, but ($\ast$) does not. To see what to do next, we look at an illustration of each of the two types of tuples, seen as points in $C_{k+1}^n$, in the case when $n=2$ and $k=2$:

$$
\begin{tikzpicture}[scale=0.9]
		\pgfmathsetmacro{\hx}{7}
		\pgfmathsetmacro{\hy}{0}
		\pgfmathsetmacro{\ix}{0.5}
		\pgfmathsetmacro{\iy}{0.5}
		\pgfmathsetmacro{\kx}{-0.5}
		\pgfmathsetmacro{\ky}{0.5}
	\foreach \h in {0,1} {
		\ifthenelse{\h=0}{\node at(\hx*\h,\hy*\h-1){Tuples allowed by ($\ast$)}}{\node at(\hx*\h,\hy*\h-1){Tuples allowed by ($\ast'$)}};		
		\foreach \i in {0,...,3} {
				\foreach \k in {0,...,3} {
					\node (\i\k) at(\hx*\h+\ix*\i+\kx*\k,\hy*\h+\iy*\i+\ky*\k){};
					\ifthenelse{\i=3}{}{\path (\hx*\h+\ix*\i+\kx*\k,\hy*\h+\iy*\i+\ky*\k) edge (\hx*\h+\ix*\i+\ix+\kx*\k,\hy*\h+\iy*\i+\iy+\ky*\k);}
					\ifthenelse{\k=3}{}{\path (\hx*\h+\ix*\i+\kx*\k,\hy*\h+\iy*\i+\ky*\k) edge (\hx*\h+\ix*\i+\kx*\k+\kx,\hy*\h+\iy*\i+\ky*\k+\ky);}
					}
			}
	}
	\tikzset{every node/.style={shape=circle,draw=black,fill=black,inner sep=1.5pt}}			\node at(0,0){};
	\node at(-0.5,0.5){};
	\node at(-1,1){};
	\node at(-1.5,1.5){};
	\node at(0,1){};
	\node at(-0.5,1.5){};
	\node at(-1,2){};
	\node at(0,2){};
	\node at(-0.5,2.5){};
	\node at(0,3){};
	\node at(0.5,2.5){};

	\node at(\hx,0){};
	\node at(-0.5+\hx,0.5){};
	\node at(-1+\hx,1){};
	\node at(-1.5+\hx,1.5){};
	\node at(\hx,1){};
	\node at(-0.5+\hx,1.5){};
	\node at(-1+\hx,2){};
	\node at(\hx,2){};
	\node at(-0.5+\hx,2.5){};
	\node at(\hx,3){};
	\node at(0.5+\hx,2.5){};
	\node at(\hx+1.5,1.5){};
	\node at(\hx+1,2){};
\end{tikzpicture}
$$

\noindent To prove the theorem it is sufficient to show what these pictures suggest: that ($\ast$) implies ($\ast'$) for individual points (tuples), and points satisfying ($\ast'$) but not ($\ast$), can never be encountered on a maximal chain whose all points satisfy ($\ast'$).
   
Suppose ($\ast$) holds for a point $(x_1,\dots,x_n)$ and $x_i\neq k+1$ for some $i\in\{1,\dots,n\}$. Assume the contrary to what ($\ast'$) requires: that $x_i>x_j$ for some $j>i$. Since $x_i\leqslant k$, we get $x_j<k$ and iteratively applying ($\ast$) results in $x_i\leqslant x_{i+1}\leqslant\dots\leqslant x_j$. This contradicts the assumption $x_i>x_j$. So ($\ast'$) follows from ($\ast$). Conversely, suppose ($\ast'$) holds. Let $x_{i+1}\in\{0,\dots,k-1\}$ and suppose unlike what ($\ast$) requires, we have $x_i>x_{i+1}$. Then ($\ast'$) forces $x_i=k+1$. Suppose the point $(x_1,\dots,x_n)$ is encountered in some maximal chain. Then at some descend along the chain there is a point $(y_1,\dots,y_n)$ with either $y_{i+1}=x_{i+1}-1$ and $y_i=x_i$, or $y_{i+1}=x_{i+1}$ and $y_i=x_i-1=k$. The second case would violate ($\ast'$). For the first case, we repeat the same argument. Eventually, we end up with the second case, and so a maximal chain containing a point that satisfies ($\ast'$) but not ($\ast$), will also contain a point that does not satisfy ($\ast'$). This completes the proof. \end{proof}

% We expect that the `orthogonal' situation, i.e., iterated stacking of rows of Figure~\ref{FigA} instead of columns, is an even less explored territory.

\section{Some conceptual remarks on stacking lattices}\label{sec remarks}

As remarked in the Introduction, the process of stacking lattices in a sequence comes from a construction of `lax colimit' in category theory. In this section we elaborate a bit on this remark. We will be concerned with the \emph{$2$-category of posets}, which we denote by $\mathbf{Pos}$. Note that this is equivalent to a subcategory of the $2$-category of categories, consisting of those categories where any two parallel morphisms are equal and any isomorphism is an identity morphism. More explicitly: 
\begin{itemize}
\item Objects in $\mathbf{Pos}$ are posets --- partially ordered sets, i.e., sets equipped with a reflexive, transitive and antisymmetric binary relation.

\item Morphisms are monotone maps between posets, i.e., maps which preserve the relation. 

\item Composition of morphisms is defined by composition of maps.

\item For any two posets $L$ and $M$, the category structure on the set of morphisms from $L$ to $M$ is a poset structure given by setting $f\leqslant g$ when $f(x)\leqslant g(x)$ for each $x\in L$. In other words, a $2$-cell between two morphisms $f$ and $g$ is a relation $f\leqslant g$ (it either does not exist, or is unique, for a given $f$ and $g$).
\end{itemize}
 
Given a sequence

$$
	\begin{tikzpicture}
	\tikzset{every node/.style={outer sep=2pt}}
	\foreach \k in {0,...,6} {
		\node (\k) at(2*\k,0){\ifthenelse{\k<5}{$M_\k$}{\ifthenelse{\k=5}{$\dots$}{\ifthenelse{\k=6}{$M_n$}{}}}};
	}

	\path [->] (0) edge node [midway,above]{$f_{0}$} (1) ;
	\path [->] (1) edge node [midway,above]{$f_{1}$} (2) ;
	\path [->] (2) edge node [midway,above]{$f_{2}$} (3) ;
	\path [->] (3) edge node [midway,above]{$f_{3}$} (4) ;
	\path [->] (4) edge node [midway,above]{$f_{4}$} (5) ;
	\path [->] (5) edge node [midway,above]{$f_{n-1}$} (6) ;
	\end{tikzpicture}
$$

\noindent of objects and morphisms in $\mathbf{Pos}$, we define its \emph{lax sum} as a specialization of the notion of a lax colimit of a diagram in a $2$-category. It is given by an object $L=\Sigma_nM_n$ in $\mathbf{Pos}$, equipped with morphisms $\iota^n_j\colon M_j\to L$, where $j\in\{0,\dots,n\}$, such that the following conditions hold:
\begin{itemize}
\item $\iota^n_j\leqslant \iota^n_{j+1}\circ f_j$ for each $j\in\{1,\dots,n-1\}$.

\item For any object $L'$ and morphisms $\iota'_j$ (where $j\in\{0,\dots,n\}$) such that $\iota'_j\leqslant \iota'_{j+1}\circ f_j$ for each $j\in\{1,\dots,n-1\}$, there exists a unique morphism $u\colon L\to L'$  such that $u\circ \iota^n_j=\iota'_j$ for each $j$.
\end{itemize}
This property can be pictured as follows:

$$
	
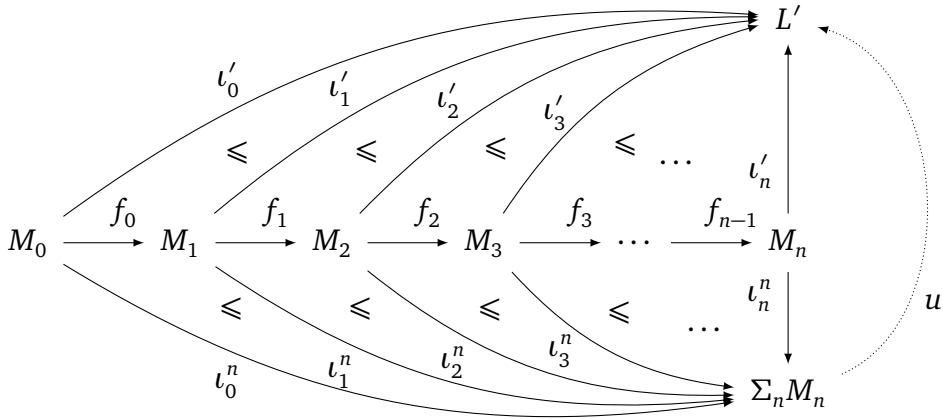
\begin{tikzpicture}
	\tikzset{every node/.style={outer sep=2pt}}
	\foreach \k in {0,...,5} {
		\node (\k) at(2*\k,0){\ifthenelse{\k<4}{$M_\k$}{\ifthenelse{\k=4}{$\dots$}{\ifthenelse{\k=5}{$M_n$}{}}}};
	}
	
	\node (L1) at (10,-2){$\Sigma_n M_n$} ;
	\node (L2) at (10,3){$L'$} ;
	
    \foreach \i in {0,...,3} {

    	\path [->] (\i) edge[bend right=20] node [near start,below]{$\iota^n_{\i}$} (L1) ;
    	\path [->] (\i) edge[bend left=20] node [near start,above]{$\iota'_{\i}$} (L2) ;
    	\draw [draw=none] (\i) -- node [near start,above]{$\hspace{\i em}\leqslant$} (L2) ;
    	\draw [draw=none] (\i) -- node [near start,below]{$\hspace{\i em}\leqslant$} (L1) ;
    }	
	
	\path [->] (0) edge node [near end,above]{$f_{0}$} (1) ;
	\path [->] (1) edge node [near end,above]{$f_{1}$} (2) ;
	\path [->] (2) edge node [near end,above]{$f_{2}$} (3) ;
	\path [->] (3) edge node [near end,above]{$f_{3}$} (4) ;
	\path [->] (4) edge node  [near end,above]{$f_{n-1}$} (5) ;
	\draw [draw=none] (4) -- node [midway,below]{\dots} (L1) ;
	\draw [draw=none] (4) -- node [near start,above]{\dots} (L2) ;
	\path [->] (5) edge node [near start,left]{$\iota^n_{n}$} (L1) ;
	\path [->] (5) edge node [near start,left]{$\iota'_{n}$} (L2) ;
	\path [->,densely dotted] (L1) edge[bend right=70] node [near start,right]{$u$} node [midway,left]{$\quad$} (L2) ;

	\end{tikzpicture}
$$

\noindent In the picture, each occurrence of the symbol $\leqslant$ stands for a $2$-cell between the composites of the surrounding diagram, indicating that the surrounding diagram \emph{lax commutes}. In other words, they represent the relations $\iota^n_{j+1}\circ f_j=\iota^n_j$ and $\iota'_{j+1}\circ f_j=\iota'_j$. 

As any object defined by a universal property, lax sum is unique up to an isomorphism: two lax sums of the same diagram will be connected by morphisms in both directions from the above universal property, which will turn out to be inverses of each other. 

Concretely, a lax sum of a sequence displayed above can be constructed as follows:
\begin{itemize}
    \item Start by taking the disjoint union of all the posets $M_j$. Let $(x,j)$ denote the representative of $x\in M_j$ in the disjoint union.
    
    \item To turn the disjoint union of posets into a poset, equip it with the relations $(x,j)\leqslant (y,j+i)$ for each $$(f_{j+i-1}\circ\dots\circ f_j)(x)\leqslant y.$$ This includes the possibility $i=0$, in which case $(f_{j+i-1}\circ\dots\circ f_i)(x)$ is defined as $x$. 
    
    \item  Each map $\iota^n_j$ is then defined by $\iota^n_j(x)=(x,j)$.
\end{itemize}
It is not difficult to show that this construction has the universal property required from a lax sum. We call this the \emph{concrete lax sum}. Note that a concrete lax sum always exists and any lax sum is canonically isomorphic to a concrete one.

There is also another (equivalent) conceptual interpretation of a lax sum. Thinking of the given sequence of posets as a functor from the chain $C_n$ seen as a category, into the category of categories, the lax sum is nothing other than the category (which in this case happens to be a poset) arising from the (dual of) standard Grothendieck construction in the theory of fibrations \cite{Gro61}. 

A sequence

$$
	\begin{tikzpicture}
	\tikzset{every node/.style={outer sep=2pt}}
	\foreach \k in {0,...,6} {
		\node (\k) at(2*\k,0){\ifthenelse{\k<5}{$M_\k$}{\ifthenelse{\k=5}{$\dots$}{\ifthenelse{\k=6}{}{}}}};
	}

	\path [->] (0) edge node [midway,above]{$f_{0}$} (1) ;
	\path [->] (1) edge node [midway,above]{$f_{1}$} (2) ;
	\path [->] (2) edge node [midway,above]{$f_{2}$} (3) ;
	\path [->] (3) edge node [midway,above]{$f_{3}$} (4) ;
	\path [->] (4) edge node [midway,above]{$f_{4}$} (5) ;
	\end{tikzpicture}
$$

\noindent of posets  and monotone maps (can be an infinite sequence or a finite one) gives rise to another such sequence, 

$$
	\begin{tikzpicture}
	\tikzset{every node/.style={outer sep=2pt}}
	\foreach \k in {0,...,6} {
		\node (\k) at(2*\k,0){\ifthenelse{\k<5}{$\Sigma_{\k}M_\k$}{\ifthenelse{\k=5}{$\dots$}{\ifthenelse{\k=6}{}{}}}};
	}

	\path [->] (0) edge node [midway,above]{$f'_{0}$} (1) ;
	\path [->] (1) edge node [midway,above]{$f'_{1}$} (2) ;
	\path [->] (2) edge node [midway,above]{$f'_{2}$} (3) ;
	\path [->] (3) edge node [midway,above]{$f'_{3}$} (4) ;
	\path [->] (4) edge node [midway,above]{$f'_{4}$} (5) ;
	\end{tikzpicture}
$$

\noindent where each $f'_n$ arises from the universal property of lax sum $\Sigma_nM_n$, as shown in the following picture:

$$
	
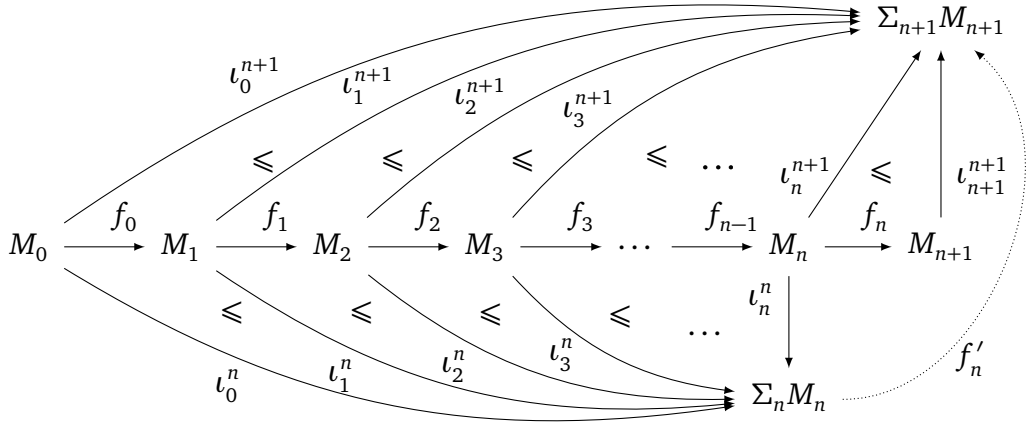
\begin{tikzpicture}
	\tikzset{every node/.style={outer sep=2pt}}
	\foreach \k in {0,...,6} {
		\node (\k) at(2*\k,0){\ifthenelse{\k<4}{$M_\k$}{\ifthenelse{\k=4}{$\dots$}{\ifthenelse{\k=5}{$M_n$}{\ifthenelse{\k=6}{$M_{n+1}$}{}}}}};
	}
	
	\node (L1) at (10,-2){$\Sigma_n M_n$} ;
	\node (L2) at (12,3){$\Sigma_{n+1} M_{n+1}$} ;
	
    \foreach \i in {0,...,3} {

    	\path [->] (\i) edge[bend right=20] node [near start,below]{$\iota^n_{\i}$} (L1) ;
    	\path [->] (\i) edge[bend left=20] node [near start,above]{$\iota^{n+1}_{\i}$} (L2) ;
    	\draw [draw=none] (\i) -- node [near start,above]{$\hspace{\i em}\leqslant$} (L2) ;
    	\draw [draw=none] (\i) -- node [near start,below]{$\hspace{\i em}\leqslant$} (L1) ;
    }	
	
	\path [->] (0) edge node [near end,above]{$f_{0}$} (1) ;
	\path [->] (1) edge node [near end,above]{$f_{1}$} (2) ;
	\path [->] (2) edge node [near end,above]{$f_{2}$} (3) ;
	\path [->] (3) edge node [near end,above]{$f_{3}$} (4) ;
	\path [->] (4) edge node  [near end,above]{$f_{n-1}$} (5) ;
	\path [->] (5) edge node  [near end,above]{$f_{n}$} (6) ;
	\draw [draw=none] (4) -- node [midway,below]{\dots} (L1) ;
	\draw [draw=none] (4) -- node [near start,above]{\dots} (L2) ;
	\draw [draw=none] (4) -- node [near start,above]{\dots} (L2) ;
	\path [->] (5) edge node [near start,left]{$\iota^n_{n}$} (L1) ;
	\path [->] (5) edge node [near start]{$\iota^{n+1}_{n}\quad\leqslant$} (L2) ;
	\path [->] (6) edge node [near start,right]{$\iota^{n+1}_{n+1}$} (L2) ;
	\path [->,densely dotted] (L1) edge[bend right=70] node [near start,right]{$f'_{n}$} node [midway,left]{$\quad$} (L2) ;

	\end{tikzpicture}
$$

\noindent As a morphism between concrete lax sums, the definition of $f'_n$ is simple: $$f'_n(x,n)=(x,n),$$
for all $x\in M_n$.

Iterating the process above gives the diagram in Figure~\ref{FigD}, where we do not distinguish in notation the $\iota$'s arising at different iterations.
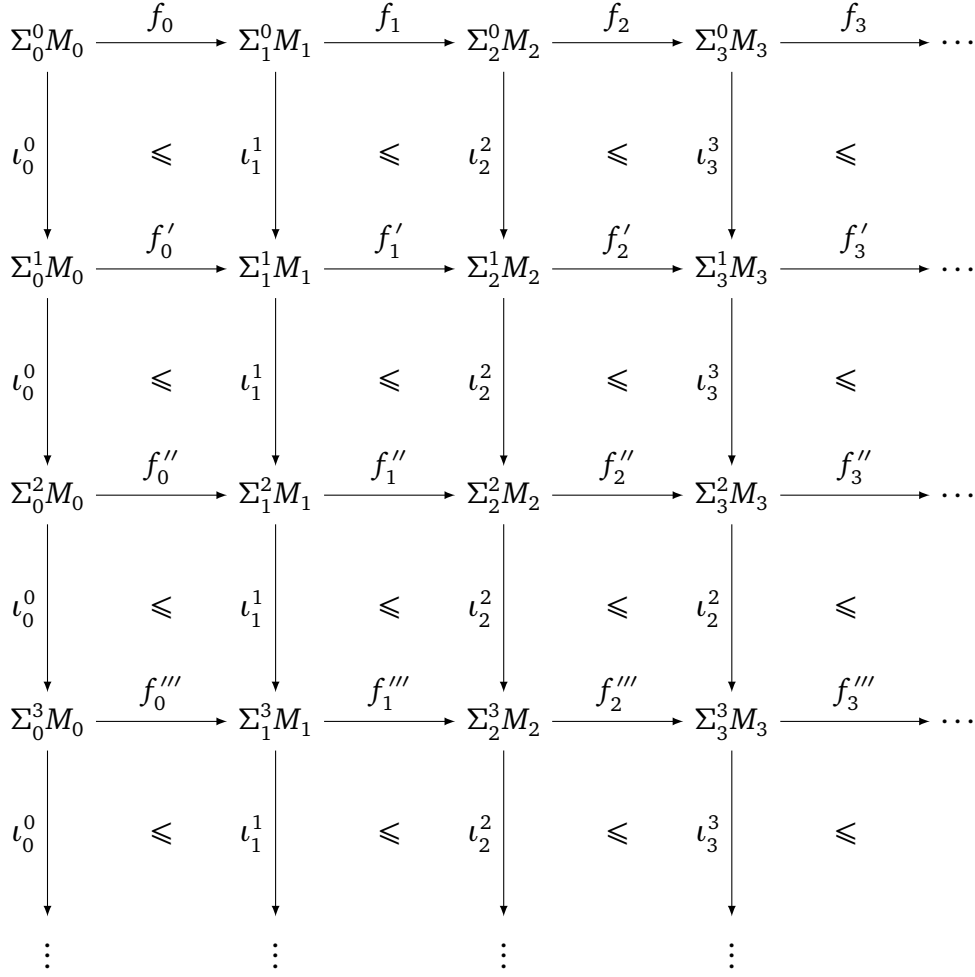
\begin{figure}
$$
\begin{tikzpicture}

		\pgfmathsetmacro{\ix}{3}
		\pgfmathsetmacro{\iy}{0}
		\pgfmathsetmacro{\jx}{0}
		\pgfmathsetmacro{\jy}{-3}
		\foreach \i in {0,...,3} {
			\foreach \j in {0,...,3} {
					\node (\i\j) at(\ix*\i+\jx*\j,\iy*\i+\jy*\j){$\Sigma^{\j}_{\i}M_{\i}$};
				}
			}
		\foreach \i in {0,...,3} {
			\foreach \j in {0,...,3} {
					\node  at(\ix*\i+\jx*\j+1.5,\iy*\i+\jy*\j-1.5){$\leqslant$};
				}
			}
		\foreach \i in {4} {
			\foreach \j in {0,...,3} {
					\node (\i\j) at(\ix*\i+\jx*\j,\iy*\i+\jy*\j){$\dots$};
				}
			}
		\foreach \i in {0,...,3} {
			\foreach \j in {4} {
					\node (\i\j) at(\ix*\i+\jx*\j,\iy*\i+\jy*\j){$\vdots$};
				}
			}
			\path [->] (00) edge node [left]{$\iota^{0}_{0}$} (01);
			\path [->] (01) edge node [left]{$\iota^{0}_{0}$} (02);
			\path [->] (02) edge node [left]{$\iota^{0}_{0}$} (03);
			\path [->] (03) edge node [left]{$\iota^{0}_{0}$} (04);
			\draw [->] (00) edge node [above]{$f_{0}$} (10);
			\draw [->] (10) edge node [above]{$f_{1}$} (20);
			\draw [->] (20) edge node [above]{$f_{2}$} (30);
			\draw [->] (30) edge node [above]{$f_{3}$} (40);

			\path [->] (10) edge node [left]{$\iota^{1}_{1}$} (11);
			\path [->] (11) edge node [left]{$\iota^{1}_{1}$} (12);
			\path [->] (12) edge node [left]{$\iota^{1}_{1}$} (13);
			\path [->] (13) edge node [left]{$\iota^{1}_{1}$} (14);
			\draw [->] (01) edge node [above]{$f'_{0}$} (11);
			\draw [->] (11) edge node [above]{$f'_{1}$} (21);
			\draw [->] (21) edge node [above]{$f'_{2}$} (31);
			\draw [->] (31) edge node [above]{$f'_{3}$} (41);
			
			\path [->] (20) edge node [left]{$\iota^{2}_{2}$} (21);
			\path [->] (21) edge node [left]{$\iota^{2}_{2}$} (22);
			\path [->] (22) edge node [left]{$\iota^{2}_{2}$} (23);
			\path [->] (23) edge node [left]{$\iota^{2}_{2}$} (24);
			\draw [->] (02) edge node [above]{$f''_{0}$} (12);
			\draw [->] (12) edge node [above]{$f''_{1}$} (22);
			\draw [->] (22) edge node [above]{$f''_{2}$} (32);
			\draw [->] (32) edge node [above]{$f''_{3}$} (42);
			
			\path [->] (30) edge node [left]{$\iota^{3}_{3}$} (31);
			\path [->] (31) edge node [left]{$\iota^{3}_{3}$} (32);
			\path [->] (32) edge node [left]{$\iota^{2}_{2}$} (33);
			\path [->] (33) edge node [left]{$\iota^{3}_{3}$} (34);
			\draw [->] (03) edge node [above]{$f'''_{0}$} (13);
			\draw [->] (13) edge node [above]{$f'''_{1}$} (23);
			\draw [->] (23) edge node [above]{$f'''_{2}$} (33);
			\draw [->] (33) edge node [above]{$f'''_{3}$} (43);

\end{tikzpicture}
$$
\caption{Iterated lax sums for a sequence $f_i\colon M_i\to M_{i+1}$ of poset morphisms}\label{FigD}
\end{figure}

We may think of $\Sigma_nM_n$ as the $n$-th `partial sum' of the infinite `series' of posets. Note that $\Sigma_0M_0=M_0$. The process of stacking lattices considered in this paper is given by this notion of `partial sum'. Recall that a lattice is a poset having binary joins and meets. The fact that in all situations considered in this paper, stacking lattices results in a lattice, comes from the following result: 

\begin{theorem}\label{ThmB} For any sequence $f_i\colon M_i\to M_{i+1}$ of posets, the following hold:
\begin{enumerate}
\item\label{for free} Each $f'_i\colon \Sigma_iM_i\to \Sigma_{i+1}M_{i+1}$ is order-reflecting (hence injective) and has down-closed image. Moreover, each $\iota^n_j\colon M_j\to \Sigma_nM_n$ is order-reflecting (and hence injective). 

\item\label{top} If $M_n$ has empty meet (top element), then so does $\Sigma_nM_n$, and $\iota^n_n\colon M_n\to \Sigma_nM_n$ preserves empty meet. 

\item\label{botpres} If each $M_i$ has empty join (bottom element) and each $f_i\colon M_i\to M_{i+1}$ preserves it, then each $\Sigma_nM_n$ has empty join and $\iota^n_0\colon M_0\to \Sigma_nM_n$, as well as each $f'_i\colon \Sigma_iM_i\to \Sigma_{i+1}M_{i+1}$, preserves empty join.

\item\label{joinpres} If each $M_i$ has binary joins and each $f_i\colon M_i\to M_{i+1}$ preserves them, then each $\Sigma_nM_n$ has binary joins, and moreover, each $f'_i\colon \Sigma_iM_i\to \Sigma_{i+1}M_{i+1}$, as well as each $\iota^n_j\colon M_j\to \Sigma_nM_n$, preserves binary joins. 

\item\label{downclose} If each $M_i$ has binary meets and each $f_i\colon M_i\to M_{i+1}$ is order-reflecting with down-closed image, then each $\Sigma_nM_n$ has binary meets and each $\iota^n_j\colon M_j\to \Sigma_nM_n$ preserves binary meets (as does each $f_i$ and $f'_i$).
\end{enumerate}
\end{theorem}

\begin{proof}
Recall that $\iota^n_j(x)=(x,j)$, and: $(x,j)\leqslant (y,k)$, when $j\leqslant k$ and $(f_{k-1}\circ\dots\circ f_j)(x)\leqslant y$. Furthermore, recall that each $f'_i\colon \Sigma_iM_i\to \Sigma_{i+1}M_{i+1}$ is defined by
$$f'_i(x,j)=(x,j),\quad j\leqslant i.$$
It is then clear that the image of each $f'_i$ is down-closed and order-reflecting, as well as that each $\iota^n_j$ is order-reflecting, so we have (\ref{for free}) --- note that an order-reflective monotone map is always injective.

When $M_n$ has top element $t$, the top element in $\Sigma_nM_n$ is given by $(t,n)$. This yields (\ref{top}).

Suppose $M_0$ has bottom element. If each $f_i$ preserves the bottom element, then $(b,0)$ is the bottom element of $\Sigma_{n}M_{n}$, where $b$ is the bottom element of $M_0$. This yields (\ref{botpres}).

Suppose now each $M_i$ has binary joins and each $f_i$ preserves them. Consider two elements $(x,j)$ and $(y,k)$ of $\Sigma_nM_n$, with $j\leqslant k$. Define
  $$(x,j)\vee (y,k)=((f_{k-1}\circ\dots\circ f_j)(x)\vee y,k),$$
where the join appearing on the right hand side of the equality is the join in $M_k$. We will now show that $(x,j)\vee (y,k)$ is the join of $(x,j)$ and $(y,k)$ in $\Sigma_nM_n$. It is easy to see that $(x,j)\leqslant (x,j)\vee (y,k)$ and $(y,k)\leqslant (x,j)\vee (y,k)$. Suppose $(x,j)\leqslant (z,l)$ and $(x,k)\leqslant (z,l)$. Then $k\leqslant l$ and 
  $$(f_{l-1}\circ\dots\circ f_j)(x)\vee (f_{l-1}\circ\dots\circ f_k)(y)\leqslant z.$$
On the other hand, since each $f$ preserves joins, 
  $$(f_{l-1}\circ\dots\circ f_k)((f_{k-1}\circ\dots\circ f_j)(x)\vee y)=(f_{l-1}\circ\dots\circ f_j)(x)\vee (f_{l-1}\circ\dots\circ f_k)(y).$$
This shows $(x,j)\vee (y,k)\leqslant (z,l)$. We have thus proved that $\Sigma_nM_n$ has binary joins. The formula for join in  $\Sigma_nM_n$ established above guarantees that each $f'_i$, as well as each $\iota^n_j$, preserve joins. This proves (\ref{joinpres}).

Suppose each $M_i$ has binary meets, and, each $f_i$ is order-reflecting with down-closed image. Then, each $f_i$ is injective and preserves binary meets. Thanks to (\ref{for free}), the same is true for each $f'_i$. We now show that each $\Sigma_nM_n$ has binary meets. Consider two elements $(x,j)$ and $(y,k)$ in $\Sigma_nM_n$. Then $x\in M_j$ and $y\in M_k$, where $j\leqslant n$ and $k\leqslant n$. Without loss of generality, assume $j\leqslant k$. Consider the meet
  $$z=(f_{k-1}\circ\dots\circ f_j)(x)\wedge y$$
in $M_k$. Let $z'$ be the unique element of $M_j$ such that   
  $$(f_{k-1}\circ\dots\circ f_j)(z')=z.$$
We will prove 
  $$(z',j)=(x,j)\wedge (y,k).$$
Note that $(z',j)\leqslant (y,k)$. 
Moreover, since $z\leqslant (f_{k-1}\circ\dots\circ f_l)(x)$, we get $z'\leqslant x$ by the order-reflection property of the  composite $f_{k-1}\circ\dots\circ f_l$. So $(z',j)\leqslant (x,j)$. Now suppose $(w,l)\leqslant (x,j)$ and $(w,l)\leqslant (y,k)$. Then 
$l\leqslant j$. Moreover, 
  $$(f_{j-1},\dots,f_l)(w)\leqslant x\textrm{ and }(f_{k-1},\dots,f_l)(w)\leqslant y.$$
This implies $(f_{k-1},\dots,f_l)(w)\leqslant z$, which by order-reflection gives 
  $$(f_{j-1},\dots,f_l)(w)\leqslant z'.$$
Then $(w,l)\leqslant(z',j)$, as desired. So $(z',j)$ is indeed the meet of $(x,j)$ and $(y,k)$. In the case when $j=k$, we get $z'=x\wedge y$, from which it follows at once that $\iota^n_j$ preserves binary meets. This proves (\ref{downclose}).
\end{proof}

The sequences of lattices that we dealt with in the paper satisfy all assumptions of the theorem above. Let us call such sequence of lattices a \emph{lattice series}. Thus, a lattice series is an infinite sequence 

$$
	\begin{tikzpicture}
	\tikzset{every node/.style={outer sep=2pt}}
	\foreach \k in {0,...,6} {
		\node (\k) at(2*\k,0){\ifthenelse{\k<5}{$M_\k$}{\ifthenelse{\k=5}{$\dots$}{\ifthenelse{\k=6}{}{}}}};
	}

	\path [->] (0) edge node [midway,above]{$f_{0}$} (1) ;
	\path [->] (1) edge node [midway,above]{$f_{1}$} (2) ;
	\path [->] (2) edge node [midway,above]{$f_{2}$} (3) ;
	\path [->] (3) edge node [midway,above]{$f_{3}$} (4) ;
	\path [->] (4) edge node [midway,above]{$f_{4}$} (5) ;
	\end{tikzpicture}
$$

\noindent of lattices and homomorphisms of join semi-lattices, such that $f_i$'s are order-reflecting and have down-closed images, and hence also preserve binary meets and are injective. A lattice series is a model of the self-dual axiomatic context for isomorphism theorems described in \cite{GosJan19}: it is an example of a `noetherian form' \cite{VNi19}. The following construction, used already in the proof of the theorem above, is related to `diagram chasing' in a noetherian form. We introduce below a short-hand notation for it.

For $x\in M_j$ and $k\geqslant j$, define
  $$x^k=\left\{\begin{array}{ll} x, & k=j,\\ (f_{k-1}\circ\dots\circ f_j)(x), & k>j.\end{array}\right.$$
For $y\in M_k$, if $y=(f_{k-1}\circ\dots\circ f_j)(x)$ for $x\in M_j$ and $j\le k$,  then this $x$, necessarily unique, will be denoted by $x=y^{-j}$.
According to Theorem~\ref{ThmB} above, the sequence of partial sums of a lattice series is again a lattice series. The proof of Theorem~\ref{ThmB} shows that the joins and the meets in the poset $\Sigma_nM_n$ can be expressed in terms of joins and meets in each $M_i$ as follows:
\begin{itemize}
    \item $(x,j)\vee(y,k)=(x^k\vee y,k)$ when $j\leqslant k$,
    \item $(x,j)\wedge(y,k)=((x^k\wedge y)^{-j},j)$ when $j\leqslant k$. 
\end{itemize}
We furthermore have the following rules; the last three of these rules are consequences of the properties of the homomorphisms $f_i$ in a lattice series:
\begin{itemize}
    \item $(x^j)^k=x^k$ when $j\leqslant k$ and $x\in M_i$, where $i\leqslant j$,

    \item $(x^{-i})^j=x^{-j}$ when $i\leqslant j\leqslant k$ and $x\in M_k$ is such that $x^{-i}$ is defined,
    
    \item $(x^{k})^{-j}=x^{j}$ when $i\leqslant j\leqslant k$ and $x\in M_i$,
    
    \item $(x^{k})^{-i}=x^{-i}$ when $i\leqslant j\leqslant k$ and $x\in M_j$ with $x^{-i}$ defined,
    
    \item $(x\wedge y)^{j}=x^{j}\wedge y^{j}$ when $x,y\in M_i$ with $i\leqslant j$,

    \item $(x\vee y)^{j}=x^{j}\vee y^{j}$ when $x,y\in M_i$ with $i\leqslant j$,
    
    \item $(x\vee y)^{-j}=x^{-j}\vee y^{-j}$ when $x,y\in M_k$ with $j\leqslant k$, and $(x\vee y)^{-j}$ is defined.
\end{itemize}
With these rules we can establish the following:

\begin{theorem}\label{ThmC}
  If in a lattice series, each lattice is a distributive lattice, i.e., the identity 
    $$x\wedge(y\vee z)=(x\wedge y)\vee (x\wedge z)$$
  holds in it, then each partial sum of the series is also a distributive lattice.
\end{theorem}
\begin{proof}
A lattice is distributive if and only if the identity \[(x\land y)\lor (y\land z)\lor(z\land x)=(x\lor y)\land(y\lor z)\land(z\lor x)\] holds in it (see e.g.~\cite{Bir48}). Since this identity is symmetric in $x,y,z$, to establish it in the partial sum of a lattice series it is sufficient to prove it for $(x,j),(y,k),(z,l)$ where $j\leqslant k\leqslant l$. This can be done as follows:
  \begin{align*}
    &((x,j)\lor(y,k))\land((y,k)\lor (z,l))\land((z,l)\lor (x,j))\\
    &=(x^k\lor y, k)\land(y^l\lor z, l)\land(z\lor x^l,l)\\
    &=(x^k\lor y, k)\land((y^l\lor z)\land(z\lor x^l), l)\\
    &=(((x^k\lor y)^l\land(y^l\lor z)\land(z\lor x^l))^{-k}, k)\\
    &=(((x^l\lor y^l)\land(y^l\lor z)\land(z\lor x^l))^{-k}, k)\\
    &=(((x^l\land y^l)\lor(y^l\land z)\lor(z\land x^l))^{-k}, k)\\
    &=((x^l\land y^l)^{-k}\lor(z\land x^l)^{-k}\lor(y^l\land z)^{-k}, k)\\
    &=(((x^l\land y^l)^{-j}\lor(z\land x^l)^{-j})^k\lor(y^l\land z)^{-k}, k)\\
    &=(((x^k\land y)^{-j}\lor(z\land x^l)^{-j})^k\lor(y^l\land z)^{-k}, k)\\
    &=((x^k\land y)^{-j}\lor(z\land x^l)^{-j}, j)\lor((y^l\land z)^{-k},k)\\
    &=((x^k\land y)^{-j}, j)\lor ((y^l\land z)^{-k}, k)\lor((z\land x^l)^{-j},j)\\
    &=((x,j)\land(y,k))\lor((y,k)\land (z,l))\lor((z,l)\land (x,j)).\qedhere
  \end{align*}
\end{proof}

All lattices in Figure~\ref{FigA} are distributive. Moreover, all maps there are join-preserving and order-reflecting, with down-closed images. So Theorems~\ref{ThmB} and \ref{ThmC} allow us to conclude that iterated stacking of these lattices along rows or columns (or in fact, along any path of consecutive arrows, for that matter), will always give rise to distributive lattices. Note that the fact that the lattices in Figure~\ref{FigA} are distributive itself follows from the theorems above: $C_m$ can be obtained by staking $C_0$ along identity morphisms;  then, $C^n_m$ is a cartesian power of a distributive lattice, and so is distributive.

\begin{remark}
The link with noetherian forms briefly mentioned above suggest to look for lattice series that arise as series of lattices of substructures of group-like structures, given by sequences of group homomorphisms. A combinatorial investigation of such lattice series, similar to what we did in the present paper for the columns of Figure~\ref{FigA}, would be interesting. These lattices would be far from being distributive, unlike the ones considered in the present paper. Lattice series of the present paper all live inside the `noetherian form of distributive lattices', studied in depth in \cite{Gra12}. The series $C^\infty_1$ can be realized as a series of substructures in another noetherian form --- that of sets and partial bijections: the lattice $C^n_1$ is of course nothing other than the lattice (Boolean algebra) of all subsets of an $n$-element set. 
\end{remark}

\section{Proof of Theorem~\ref{ThmD}}\label{sec proof}

The proof of Theorem~\ref{ThmD} that we present below relies on further analysis of the concept of lax sum. We will show first that each of the squares in Figure~\ref{FigD} is a `co-comma' diagram of posets seen as categories, which we simply call a `lax pushout'; note however that it is not an instance of a lax colimit --- rather, it is an instance of (another) type of `indexed colimit' (also called a `weighted colimit') in a $2$-category \cite{Kel05}.  This will provide another way of constructing the posets $\Sigma^k_nM_n$: the morphisms $\iota^0_0$ are identity morphisms and so we can start by forming the top left square, the one next to it, and so on in each row successively. 

A diagram

$$
\begin{tikzpicture}
	\node (K) at(0,2){$K$};
	\node (M) at(2,2){$M$};
	\node (L) at(2,0){$L$};
	\node (N) at(0,0){$N$};

	\path [->] (K) edge node [above]{$f$} (M);
	\path [->] (M) edge node [right]{$g'$} (L);
	\path [->] (K) edge node [left]{$g$} (N);
	\path [->] (N) edge node [below]{$f'$} (L);
\end{tikzpicture}
$$ 

\noindent in $\mathbf{Pos}$ is a \emph{lax pushout} if and only if: 
\begin{itemize}
\item $f'$ and $g'$ are order-reflecting morphisms of posets (hence injective) whose images are disjoint, with the union of images giving the entire $L$,

\item the image of $f'$ is down-closed and for any $a\in N$ and $b\in M$, we have $f'(a)\leqslant g'(b)$ if and only if there exists $c\in K$ such that $a\leqslant g(c)$ and $f(c)\leqslant b$.
\end{itemize}
It is easy to see that such a diagram has the following universal property (that is a specialization of the construction of a co-comma object in a general $2$-category):
\begin{itemize}
\item the square lax commutes, i.e., there is a $2$-cell from $f'\circ g$ to $g'\circ f$ (thus, $f'\circ g\leqslant g'\circ f$),

\item for any other lax commuting square over $f,g$, 

$$
\begin{tikzpicture}
	\node (K) at(0,2){$K$};
	\node (M) at(2,2){$M$};
	\node (L) at(2,0){$L'$};
	\node (N) at(0,0){$N$};

	\path [->] (K) edge node [above]{$f$} (M);
	\path [->] (M) edge node [right]{$g''$} (L);
	\path [->] (K) edge node [left]{$g$} (N);
	\path [->] (N) edge node [below]{$f''$} (L);
	\draw [->,draw=none] (K) -- node {$\leqslant$} (L);
\end{tikzpicture}
$$

\noindent there is a unique morphism $u$ such that $u\circ f'=f''$ and $u\circ g'=g''$, as shown on the picture:
$$
\begin{tikzpicture}
	\node (K) at(0,2){$K$};
	\node (M) at(2,2){$M$};
	\node (L) at(2,0){$L$};
	\node (N) at(0,0){$N$};
	\node (R) at(3,-1){$L'$};

	\path [->] (K) edge node [above]{$f$} (M);
	\path [->] (M) edge node [right]{$g'$} (L);
	\path [->] (K) edge node [left]{$g$} (N);
	\path [->] (N) edge node [below]{$f'$} (L);
	\draw [->,draw=none] (K) -- node {$\leqslant$} (L);
	\path [->,densely dotted] (L) edge node [above]{$u$} (R);
	\path [->] (N) edge[bend right=30] node [below]{$f''$} (R) ;
	\path [->] (M) edge[bend left=30] node [right]{$g''$} (R) ;
\end{tikzpicture}
$$
\end{itemize}
Just as any other (weighted) colimit, a lax pushout is unique up to a canonical isomorphism. It exists for any given $f$ and $g$ as we can construct one \emph{concretely} as follows:
\begin{itemize}
\item Let the elements of $L$ be pairs of the form $(a,1)$, where $a\in N$, as well as pairs of the form $(b,2)$, where $b\in M$.

\item Define $(a,1)\leqslant (a',1)$ when $a\leqslant a'$ in $N$.

\item Define $(b,2)\leqslant (b',2)$ when $b\leqslant b'$ in $N$.

\item Define $(a,1)\leqslant (b,2)$ when there exists $c\in K$ such that $a\leqslant g(c)$ and $f(c)\leqslant b$.

\item Define $f'$ and $g'$ by the identities $f'(a)=(a,1)$ and $g'(b)=(b,2)$.
\end{itemize}

\begin{lemma}\label{Lem lax pushout matrix}
Each square in Figure~\ref{FigD} is a lax pushout.
\end{lemma}

\begin{proof}
It suffices to prove this for the squares in the first row, since the rest of the rows are obtained by iteration. So we prove that for each $n$, the following square is a lax pushout:

$$
\begin{tikzpicture}
	\node (K) at(0,3){$M_n$};
	\node (M) at(3,3){$M_{n+1}$};
	\node (L) at(3,0){$\Sigma_{n+1} M_{n+1}$};
	\node (N) at(0,0){$\Sigma_n M_n$};

	\path [->] (K) edge node [above]{$f_n$} (M);
	\path [->] (M) edge node [right]{$\iota^{n+1}_{n+1}$} (L);
	\path [->] (K) edge node [left]{$\iota^n_n$} (N);
	\path [->] (N) edge node [below]{$f'_n$} (L);
	\draw [->,draw=none] (K) -- node {$\leqslant$} (L);
\end{tikzpicture}
$$ 

\noindent For this, we need to check the following:
\begin{itemize}
\item[(i)] $f'_n$ an order-preserving morphism of posets (and hence is injective), having down-closed image.

\item[(ii)] $\iota^{n+1}_{n+1}$ is an order-reflecting morphism of posets (and hence injective).

\item[(iii)] The images of $f'_n$ and $\iota^{n+1}_{n+1}$ are disjoint, and their union is the entire $\Sigma_{n+1}M_{n+1}$.

\item[(iv)] For any $a\in \Sigma_nM_n$ and $b\in M_{n+1}$, we have $f'_n(a)\leqslant \iota^{n+1}_{n+1}(b)$ if and only if there exists $c\in K$ such that $a\leqslant \iota^{n}_n(c)$ and $f_n(c)\leqslant b$.
\end{itemize}
We know (i) and (ii) already from Theorem~\ref{ThmB}. Since $f'_n$ is defined by $f'_n(x,j)=(x,j)$, where $j\leqslant n$ and $x\in M_j$, and $\iota^{n+1}_{n+1}(x)=(x,n+1)$ where $x\in M_{n+1}$, the images of these two functions are clearly disjoint. Their union is the entire $\Sigma_{n+1}M_{n+1}$ since the latter only consists of elements of the form $(x,j)$, where $x\in M_j$ with $j\leqslant n+1$. So we have (iii). For any $(x,j)\in \Sigma_nM_n$ and $y\in M_{n+1}$, we have $f'_n(x,j)\leqslant \iota^{n+1}_{n+1}(y)$ if and only if $(x,j)\leqslant (y,n+1)$. This is the case if and only if $x^{n+1}\leqslant y$ (see the previous section for this notation). If $x^{n+1}\leqslant y$, then define $c=x^n\in M_n$. We get $(x,j)\leqslant (c,n)=\iota^n_n(c)$ and $f_n(c)=x^{n+1}\leqslant y$. Conversely, if there exists $c\in M_n$ such that $(x,j)\leqslant \iota^n_n(c)=(c,n)$ and $f_n(c)\leqslant y$, then $x^n\leqslant c$ and so $x^{n+1}=f_n(x^n)\leqslant f_n(c)\leqslant y$, which gives $(x,j)\leqslant (y,n+1)$. The proof is now complete.
\end{proof}

We are now ready to prove Theorem~\ref{ThmD}.

First, we show the following (which, actually, does not require any of what we have done above): 

\begin{lemma}
The elements of $C^n_{k+m}$ satisfying ($\ast$) are closed under meets and joins in $C^n_{k+m}$. 
\end{lemma}

\begin{proof}
Suppose both $(x_1,\dots,x_n)$ and $(x'_1,\dots,x'_n)$ satisfy ($\ast$). Their meet in $C^n_{k+m}$ is given by 
$$(x_1,\dots,x_n)\wedge (x'_1,\dots,x'_n)=(\min(x_1,x'_1),\dots,\min(x_n,x'_n)).$$
Suppose $\min(x_{i+1},x'_{i+1})\in \{1,\dots,k-1\}$. Without loss of generality, assume  
$x_{i+1}\le x'_{i+1}$.
Then $x_i\leqslant x_{i+1}$ and so
$$\min(x_i,x'_i)\leqslant x_{i+1}=\min(x_{i+1},x'_{i+1}).$$ This proves that tuples satisfying ($\ast$) are closed under meets. The join in $C^n_{k+m}$ is given by
$$(x_1,\dots,x_n)\vee (x'_1,\dots,x'_n)=(\max(x_1,x'_1),\dots,\max(x_n,x'_n)).$$
Suppose $\max(x_{i+1},x'_{i+1})\in \{1,\dots,k-1\}$. Then both $x_{i+1}$ and $x'_{i+1}$ belong to $\{1,\dots,k-1\}$ and the conclusion
  \[\max(x_i,x'_i)\leqslant\max(x_{i+1},x'_{i+1})\]
is immediate by the fact that both tuples satisfy ($\ast$).
\end{proof}

To prove Theorem~\ref{ThmD}, it remains to show that each poset $\Sigma^k_nC^n_m$ is isomorphic to the sublattice of $C^n_{k+m}$ consisting of $n$-tuples satisfying ($\ast$). Consider the diagram 

$$\begin{tikzpicture}
		\pgfmathsetmacro{\ix}{1.7}
		\pgfmathsetmacro{\iy}{0}
		\pgfmathsetmacro{\jx}{0}
		\pgfmathsetmacro{\jy}{-1.7}
		\foreach \i in {0,...,3} {
			\foreach \j in {0,...,3} {
					\node (\i\j) at(\ix*\i+\jx*\j,\iy*\i+\jy*\j){$C^{\i}_{\j+m}$};
				}
			}
		\foreach \i in {0,...,2} {
			\foreach \j in {0,...,2} {
					\node  at(\ix*\i+\jx*\j+0.8,\iy*\i+\jy*\j-0.85){$\leqslant$};
				}
			}
		\foreach \i in {4} {
			\foreach \j in {0,...,3} {
					\node (\i\j) at(\ix*\i+\jx*\j,\iy*\i+\jy*\j){$\dots$};
				}
			}
		\foreach \i in {0,...,3} {
			\foreach \j in {4} {
					\node (\i\j) at(\ix*\i+\jx*\j,\iy*\i+\jy*\j){$\vdots$};
				}
			}
		\foreach \i in {0,...,3} {
			\draw [->] (\i0) -- (\i1);
			\draw [->] (\i1) -- (\i2);
			\draw [->] (\i2) -- (\i3);
			\draw [->] (\i3) -- (\i4);
			\draw [->] (0\i) -- (1\i);
			\draw [->] (1\i) -- (2\i);
			\draw [->] (2\i) -- (3\i);
			\draw [->] (3\i) -- (4\i);

		}
\end{tikzpicture}$$

\noindent of lattices $C^n_{k+m}$, where each vertical arrow $C^n_{k+m}\to C^{n}_{k+1+m}$ maps $(x_1,\dots,x_n)$ to $(x_1+1,\dots,x_n+1)$ and each horizontal arrow $C^n_{k+m}\to C^{n+1}_{k+m}$ maps $(x_1,\dots,x_n)$ to $(0,x_1,\dots,x_n)$. Each square in this diagram lax commutes: left-bottom composite is below the top-right composite. Indeed:
$$(0,x_1+1,\dots,x_n+1)\leqslant (1,x_1+1,\dots,x_n+1).$$
It is easy to see that both the horizontal and the vertical morphisms preserve the property ($\ast$). So the diagram above restricts to the following diagram, where $C^{n\ast}_{k+m}$ denotes the sublattice of $C^n_{k+m}$ determined by the property ($\ast$), while the morphisms are defined in the same way as above:

$$\begin{tikzpicture}
		\pgfmathsetmacro{\ix}{1.7}
		\pgfmathsetmacro{\iy}{0}
		\pgfmathsetmacro{\jx}{0}
		\pgfmathsetmacro{\jy}{-1.7}
		\foreach \i in {0,...,3} {
			\foreach \j in {0,...,3} {
					\node (\i\j) at(\ix*\i+\jx*\j,\iy*\i+\jy*\j){$C^{\i\ast}_{\j+m}$};
				}
			}
		\foreach \i in {0,...,2} {
			\foreach \j in {0,...,2} {
					\node  at(\ix*\i+\jx*\j+0.8,\iy*\i+\jy*\j-0.85){$\leqslant$};
				}
			}
		\foreach \i in {4} {
			\foreach \j in {0,...,3} {
					\node (\i\j) at(\ix*\i+\jx*\j,\iy*\i+\jy*\j){$\dots$};
				}
			}
		\foreach \i in {0,...,3} {
			\foreach \j in {4} {
					\node (\i\j) at(\ix*\i+\jx*\j,\iy*\i+\jy*\j){$\vdots$};
				}
			}
		\foreach \i in {0,...,3} {
			\draw [->] (\i0) -- (\i1);
			\draw [->] (\i1) -- (\i2);
			\draw [->] (\i2) -- (\i3);
			\draw [->] (\i3) -- (\i4);
			\draw [->] (0\i) -- (1\i);
			\draw [->] (1\i) -- (2\i);
			\draw [->] (2\i) -- (3\i);
			\draw [->] (3\i) -- (4\i);

		}
\end{tikzpicture}$$

\noindent Lemma~\ref{Lem lax pushout matrix} and the fact that lax pushouts are unique up to a canonical isomorphism, reduces the proof of Theorem~\ref{ThmD} to showing that each square the in the above diagram is a lax pushout. Indeed, for if it is so, we will be able to recursively create isomorphisms $d_{k,n,m}\colon\Sigma^k_n C^n_m\to C^{n\ast}_{k+m}$, the diagonal arrows in the following diagram,

$$
\begin{tikzpicture}
		\pgfmathsetmacro{\ix}{2}
		\pgfmathsetmacro{\iy}{0}
		\pgfmathsetmacro{\jx}{0}
		\pgfmathsetmacro{\jy}{-2}
		\foreach \i in {0,...,3} {
			\foreach \j in {0,...,3} {
					\node (A\i\j) at(\ix*\i+\jx*\j,\iy*\i+\jy*\j){$\Sigma^{\j}_{\i}C^{\i}_{m}$};
				}
			}
		\foreach \i in {0,...,2} {
			\foreach \j in {0,...,2} {
					\node  at(\ix*\i+\jx*\j+0.8,\iy*\i+\jy*\j-0.85){};
				}
			}
		\foreach \i in {4} {
			\foreach \j in {0,...,3} {
					\node (A\i\j) at(\ix*\i+\jx*\j,\iy*\i+\jy*\j){$\dots$};
				}
			}
		\foreach \i in {0,...,3} {
			\foreach \j in {4} {
					\node (A\i\j) at(\ix*\i+\jx*\j,\iy*\i+\jy*\j){$\vdots$};
				}
			}
		
		\pgfmathsetmacro{\ix}{2}
		\pgfmathsetmacro{\iy}{0}
		\pgfmathsetmacro{\jx}{0}
		\pgfmathsetmacro{\jy}{-2}
		\foreach \i in {0,...,3} {
			\foreach \j in {0,...,3} {
					\node (B\i\j) at(1+\ix*\i+\jx*\j,1+\iy*\i+\jy*\j){$C^{\i\ast}_{\j+m}$};
				}
			}
		\foreach \i in {0,...,2} {
			\foreach \j in {0,...,2} {
					\node  at(1+\ix*\i+\jx*\j+0.8,1+\iy*\i+\jy*\j-0.85){};
				}
			}
		\foreach \i in {4} {
			\foreach \j in {0,...,3} {
					\node (B\i\j) at(1+\ix*\i+\jx*\j,1+\iy*\i+\jy*\j){$\dots$};
				}
			}
		\foreach \i in {0,...,3} {
			\foreach \j in {4} {
					\node (B\i\j) at(1+\ix*\i+\jx*\j,1+\iy*\i+\jy*\j){$\vdots$};
				}
			}
			
		\foreach \i in {0,...,3} {
			\draw [->] (A\i0) -- (A\i1);
			\draw [->] (A\i1) -- (A\i2);
			\draw [->] (A\i2) -- (A\i3);
			\draw [->] (A\i3) -- (A\i4);
			\draw [->] (A0\i) -- (A1\i);
			\draw [->] (A1\i) -- (A2\i);
			\draw [->] (A2\i) -- (A3\i);
			\draw [->] (A3\i) -- (A4\i);

		}
		\foreach \i in {0,...,3} {
			\draw [->] (B\i0) -- (B\i1);
			\draw [->] (B\i1) -- (B\i2);
			\draw [->] (B\i2) -- (B\i3);
			\draw [->] (B\i3) -- (B\i4);
			\draw [->] (B0\i) -- (B1\i);
			\draw [->] (B1\i) -- (B2\i);
			\draw [->] (B2\i) -- (B3\i);
			\draw [->] (B3\i) -- (B4\i);

		}
		
		\foreach \i in {0,...,3} {
    		\foreach \j in {0,...,3} {
			\draw [->] (A\i\j) -- (B\i\j);
			}
		}
\end{tikzpicture}
$$

\noindent starting with identity morphisms along the top and the left side of the picture: note that when $n=0$ or $k=0$, we have  $\Sigma^k_n C^n_m=C^{n\ast}_{k+m}$, and so we can set the $d_{0,n,m}$ and $d_{k,0,m}$ diagonal morphisms to be the identity morphisms.

So the final step in the proof is given by the following:

\begin{lemma}\label{Lem lax pushout}
For any $k,n,m$, the following square is a lax pushout:
$$
\begin{tikzpicture}
	\node (K) at(0,3){$C^{n\ast}_{k+m}$};
	\node (M) at(3,3){$C^{n+1\ast}_{k+m}$};
	\node (L) at(3,0){$C^{n+1\ast}_{k+1+m}$};
	\node (N) at(0,0){$C^{n\ast}_{k+1+m}$};

	\node (KK) at(-3,4){$(x_1,\dots,x_n)$};
	\node (MM1) at(3,4){$(0,x_1,\dots,x_n)$};
	\node (MM2) at(6,3){$(y_1,\dots,y_{n+1})$};
	\node (NN1) at(-3,0){$(x_1+1,\dots,x_n+1)$};
	\node (NN2) at(0,-1){$(z_1,\dots,z_n)$};

	\node (LL1) at(3,-1){$(0,z_1,\dots,z_n)$};
	\node (LL2) at(6,0){$(y_1+1,\dots,y_{n+1}+1)$};

	\path [->] (K) edge node [above]{} (M);
	\path [->] (M) edge node [right]{} (L);
	\path [->] (K) edge node [left]{} (N);
	\path [->] (N) edge node [below]{} (L);
	\draw [->,draw=none] (K) -- node {} (L);

	\path [|->] (KK) edge node{} (MM1);
	\path [|->] (KK) edge node{} (NN1);
	\path [|->] (MM2) edge node{} (LL2);
	\path [|->] (NN2) edge node{} (LL1);
\end{tikzpicture}
$$ 
\end{lemma}

\begin{proof}
The bottom map is clearly an order-reflecting morphism of posets, and it is easy to see that its image is down-closed. The right map is also obviously an order-reflecting morphism of posets. The images of these two maps certainly do not intersect. Consider any tuple $(x_1,\dots,x_{n+1})$ in $C^{n+1\ast}_{k+1+m}$. If $x_1=0$, then it belongs to the image of the bottom map. If $x_1\neq 0$, then by the property $(\ast)$, we can never have $x_j=0$ for any other coordinate and so $(x_1,\dots,x_{n+1})=(y_1+1,\dots,y_{n+1}+1)$ for some tuple  $(y_1,\dots,y_{n+1})$ in $C^{n+1}_{k+m}$. We need to check that $(y_1,\dots,y_{n+1})$ satisfies ($\ast$). Let $y_{i+1}\in\{0,\dots,k-1\}$. Then $y_{i+1}+1\in\{0,\dots,k\}$. Since $(y_1+1,\dots,y_{n+1}+1)$ satisfies ($\ast$), we have $y_{i}+1\leqslant y_{i+1}+1$, and so $y_{i}\leqslant y_{i+1}$. This proves $(y_1,\dots,y_{n+1})\in C^{n+1\ast}_{k+m}$, showing that every element in the bottom right lattice falls in the image of one of the maps going into it. Complete the proof that the square is a lax pushout, it remains to show that $(0,z_1,\dots,z_n)\leqslant (y_1+1,\dots,y_{n+1}+1)$ for two $(n+1)$-tuples in $C^{n+1\ast}_{k+1+m}$ if and only if 
$$(z_1,\dots,z_n)\leqslant (x_1+1,\dots,x_{n}+1)\textrm{ and }(0,x_1,\dots,x_{n})\leqslant (y_1,\dots,y_{n+1})$$
for some $n$-tuple $(x_1,\dots,x_{n})$ in $C^{n\ast}_{k+m}$. The `if' part is easy to see (it follows from the fact that the square is lax commuting). To show the `only if' part, assume $(0,z_1,\dots,z_n)\leqslant (y_1+1,\dots,y_{n+1}+1)$ in  $C^{n+1\ast}_{k+1+m}$. Then the $(n+1)$-tuple $(y_1,y_2,\dots,y_{n+1})$ belongs to $C^{n+1\ast}_{k+m}$, which easily implies that the $n$-tuple $(y_2,\dots,y_{n})$ belongs to $C^{n\ast}_{k+m}$. This is the desired $n$-tuple, since 
 $$(z_1,\dots,z_n)\leqslant (y_2+1,\dots,y_{n}+1)\textrm{ and }(0,y_2,\dots,y_{n})\leqslant (y_1,\dots,y_{n+1}).$$
This completes the proof. 
\end{proof}

\section{The rows of Figure~\ref{FigA}}\label{Sec rows}

In this section we briefly consider the `orthogonal' situation, i.e., iterated stacking of rows of Figure~\ref{FigA} instead of columns. Theorem~\ref{ThmD} has the following analogue in this case.

\begin{figure}
$$\begin{array}{rl}
 & m=0,1,2,3,4,5,\dots\\\hline
n=0: \\\hline
k=0 & 1, 1, 1, 1, 1, 1, \dots \textrm{(A000012)}\\
k=1 & 1, 1, 1, 1, 1, 1, \dots \textrm{(A000012)}\\
k=2 & 1, 1, 2, 5, 14, 42, \dots \textrm{(A000108 -- Catalan numbers)}\\
k=3 & 1, 1, 5, 42, 462, 6006, \dots \textrm{(A005789 -- $3$-dimensional Catalan numbers)}\\
k=4 & 1, 1, 14, 462, 24024, 1662804, \dots \textrm{(A005790 -- $4$-dimensional Catalan numbers)}\\\hline
n=1:\\\hline
k=0 & 1, 1, 1, 1, 1, 1, \dots \textrm{(A000012)}\\
k=1 & 1, 1, 2, 15, 14, 42, \dots \textrm{(A000108 -- Catalan numbers)}\\
k=2 & 1, 1, 5, 42, 462, 6006, \dots \textrm{(A005789 -- $3$-dimensional Catalan numbers)}\\
k=3 & 1, 1, 14, 462, 24024, 1662804, \dots \textrm{(A005790 -- $4$-dimensional Catalan numbers)}\\
k=4 & 1, 1, 42, 6006, 1662804, 701149020, \dots \textrm{(A005791 -- $5$-dimensional Catalan numbers)}\\\hline
n=2:\\\hline
k=0 & 1, 2, 6, 20, 70, 252, \dots \textrm{(A000680 -- $(2m)!/(m!)^2$)}\\
k=1 & 1, 2, 16, 192, 2816, 46592, \dots \textrm{(A006335 -- Kreweras)}\\
k=2 & 1, 2, 46, 2240, 160504, 14594568, \dots \textrm{(no entry)}\\
k=3 & 1, 2, 140, 30108, 11721144, 6625780016, \dots \textrm{(no entry)}\\
k=4 & 1, 2, 444, 448272, 1024045836, 3936970992944, \dots \textrm{(no entry)}\\\hline
n=3:\\\hline
k=0 & 1, 6, 90, 1680, 34650, 756756, \dots \textrm{(A006480 -- $(3m)!/(m!)^3$)}\\
k=1 & 1, 6, 288, 24444, 2738592, 361998432, \dots \textrm{(no entry)}\\
k=2 & 1, 6, 918, 363984, 234506712, 203517798360, \dots \textrm{(no entry)}\\
k=3 & 1, 6, 2988, 5753484, 22547430432, 137927632096368, \dots \textrm{(no entry)}\\
k=4 & 1, 6, 9936, 96198840, 2404039625820, 109858268535649608, \dots \textrm{(no entry)}\\\end{array}$$
\caption{Maximal chain numbers $\#\Sigma^k_m C^n_m$ and the corresponding entry code (indicated in brackets) in the On-Line Encyclopedia of Integer Sequences (accessed 29/09/2020).}\label{FigC}
\end{figure}

\begin{theorem}\label{ThmE}
  The poset $\Sigma_m^kC_m^n$ is isomorphic to the sublattice $C_m^{n+k\star}$ of $C_m^{n+k}$ consisting of those elements that satisfy
  \begin{enumerate}
    \item[($\star$)] for each $i\in\{1,\dots,n\}$, $x_i\leqslant x_{n+1}\leqslant x_{n+2}\leqslant\dots\leqslant x_{n+k}$.
  \end{enumerate}
\end{theorem}

The proof of this theorem can be established along the lines of the proof of Theorem~\ref{ThmD}, where the lax pushout from Lemma~\ref{Lem lax pushout} would now be replaced with the following lax pushout:
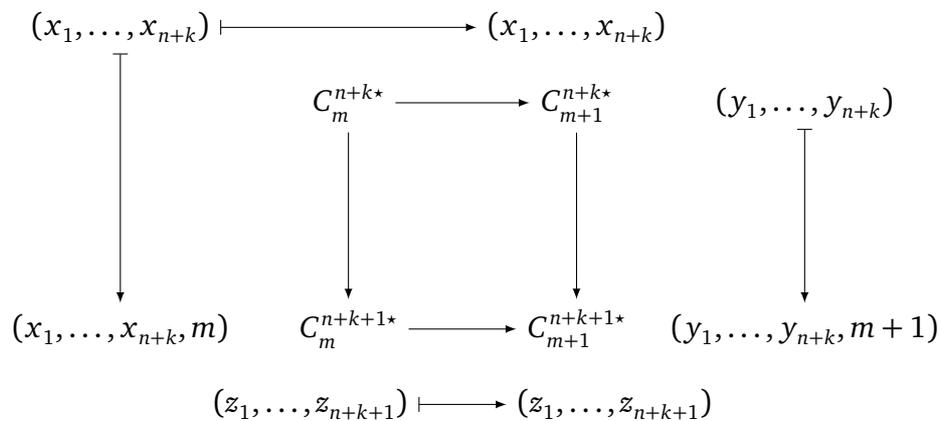
\[\begin{tikzpicture}
	\node (K) at(0,3){$C^{n+k\star}_{m}$};
	\node (M) at(3,3){$C^{n+k\star}_{m+1}$};
	\node (L) at(3,0){$C^{n+k+1\star}_{m+1}$};
	\node (N) at(0,0){$C^{n+k+1\star}_{m}$};

	\node (KK) at(-3,4){$(x_1,\dots,x_{n+k})$};
	\node (MM1) at(3,4){$(x_1,\dots,x_{n+k})$};
	\node (MM2) at(6,3){$(y_1,\dots,y_{n+k})$};
	\node (NN1) at(-3,0){$(x_1,\dots,x_{n+k}, m)$};
	\node (NN2) at(-.5,-1){$(z_1,\dots,z_{n+k+1})$};

	\node (LL1) at(3.5,-1){$(z_1,\dots,z_{n+k+1})$};
	\node (LL2) at(6,0){$(y_1,\dots,y_{n+k}, m+1)$};

	\path [->] (K) edge node [above]{} (M);
	\path [->] (M) edge node [right]{} (L);
	\path [->] (K) edge node [left]{} (N);
	\path [->] (N) edge node [below]{} (L);
	\draw [->,draw=none] (K) -- node {} (L);

	\path [|->] (KK) edge node{} (MM1);
	\path [|->] (KK) edge node{} (NN1);
	\path [|->] (MM2) edge node{} (LL2);
	\path [|->] (NN2) edge node{} (LL1);
\end{tikzpicture}\]
This representation recovers pictures from the Introduction in the case when $k=1$ and $n=1,2$ --- see Figures~\ref{Fig row 2} and \ref{Fig row 3}.

\begin{figure}
\input{Row2}
\caption{Lattices $\Sigma^k_mC^1_m$ for $k\in\{0,1,2\}$ and $m\in\{0,1,2,3,4\}$.}\label{Fig row 2}
\end{figure}

\begin{figure}
\input{Krewerasstack}
\caption{Lattices $\Sigma^k_mC^2_m$ for $k\in\{0,1,2\}$ and $m\in\{0,1,2,3,4\}$.}\label{Fig row 3}
\end{figure} 

Maximal chains in the poset $\Sigma_mC_m^2$ correspond to `lattice walks' on a cartesian plane of length $3m$ starting and ending at $(0,0)$, remaining in the first quadrant and using only the NE, W, and S steps, that is, $(1,1)$, $(-1,0)$ and $(0,-1)$ steps --- this is the Kreweras situation mentioned in the Introduction (see \cite{BouMel05,Kre65}). More generally, we have:

\begin{theorem}
  There is a bijection between the set of maximal chains in the poset $\Sigma_mC_m^n$ and the set of walks of length $(n+1)m$ in $C^n_m$, starting and ending at $(0,\dots,0)$, and using only $(1,1,\dots, 1)$, $(-1,0,\dots, 0)$, $(0,-1,\dots, 0)$, \ldots, and $(0,0,\dots,-1)$ steps.
\end{theorem}
\begin{proof}[Proof (sketch)]
  Using the representation $\Sigma_mC_m^n\approx C_m^{n+1\star}$ of Theorem~\ref{ThmE}, the desired bijection is given by: 
  \begin{itemize}
      \item for $1\leqslant i\leqslant n$, associating an increase in the $i$-th coordinate of a maximal chain in $C_m^{n+1\star}$ with the step in a walk given by $-1$ in the $i$-th position, 
      
      \item and associating an increase in the $(n+1)$-th coordinate with the diagonal step $(1,1,\dots,1)$ in a walk.  
      \end{itemize}
\end{proof}

The first few terms of the sequence of maximal chain numbers for the stacking of lattices along the rows of Figure~\ref{FigA} are shown in Figure~\ref{FigC}. 

\bigskip
\textbf{Acknowledgement.}
The authors would like to thank Sarah Selkirk for taking an interest in an early version of this paper and insightful discussions that followed, as well as some useful suggestions on the presentation.

\end{document}